\documentclass[a4paper,12pt]{amsart}
\usepackage[utf8x]{inputenc}
\usepackage{amssymb}
\usepackage[usesnames,svgnames]{xcolor}
\usepackage{tikz,pgf}
  \usetikzlibrary{arrows}
\usepackage{amsrefs}
\usepackage{subfigure}
\usepackage{longtable}
\usepackage{multirow}
\usepackage{prettyref}
\usepackage[colorlinks=true,
	    urlcolor=MidnightBlue,
	    citecolor=DarkGreen]{hyperref}
\newrefformat{eq}{(\ref{#1})}
\newrefformat{fig}{Figure~\ref{#1}}
\newrefformat{tab}{Table~\ref{#1}}
\newrefformat{sec}{Section~\ref{#1}}
\newrefformat{app}{Appendix~\ref{#1}}
\newrefformat{def}{Definition~\ref{#1}}
\newrefformat{thm}{Theorem~\ref{#1}}
\newrefformat{prop}{Proposition~\ref{#1}}
\newrefformat{lem}{Lemma~\ref{#1}}
\newrefformat{rem}{Remark~\ref{#1}}
  \newtheorem{theorem}{Theorem}[section]
  \newtheorem{proposition}[theorem]{Proposition}
  \newtheorem{lemma}[theorem]{Lemma}
  \newtheorem{definition}[theorem]{Definition}
\theoremstyle{remark}
  \newtheorem{remark}[theorem]{Remark}
\newcommand{\pleq}{\leq_{P}}

\newcommand{\parr}{\leftarrow_{P}}
\newcommand{\qarr}{\leftarrow_{Q}}
\newcommand{\KK}{\mathbb{K}}
\newcommand{\CC}{\mathbb{C}}
\newcommand{\ac}{\mathfrak{a}}
\newcommand{\cc}{\mathfrak{c}}
\newcommand{\CL}{\underline{\mathfrak{B}}}
\author{Henri M\"uhle}
\address{Fak. f\"ur Mathematik, Universit\"at Wien, Garnisongasse 3, 1090 Wien, Austria}
\email{henri.muehle@univie.ac.at}
\thanks{This work was funded by the FWF research grant no. Z130-N13.}
\newcommand{\chainContext}{\scriptsize
	\begin{tikzpicture}
		\draw(0,0) node{
			\begin{tabular}{|c||c|c|c|c|}
				\hline
				$\leq$ & $a$ & $b$ & $c$ & $d$\\
				\hline\hline
				$a$ & $\times$ & $\times$ & $\times$ & $\times$\\
				\hline
				$b$ & & $\times$ & $\times$ & $\times$\\
				\hline
				$c$ & & & $\times$ & $\times$\\
				\hline
				$d$ & & & & $\times$\\
				\hline
			\end{tabular}};
		\draw(4,0) node{
			\begin{tikzpicture}
				\def\r{.8};
				\draw(1,1) node[circle,draw,scale=\r,label=below right:$a$](n0){};
				\draw(1,2) node[circle,draw,scale=\r,label=below right:$b$](n1){};
				\draw(1,3) node[circle,draw,scale=\r,label=below right:$c$](n2){};
				\draw(1,4) node[circle,draw,scale=\r,label=below right:$d$](n3){};
				\draw (n0) -- (n1) -- (n2) -- (n3);
			\end{tikzpicture}};
	\end{tikzpicture}
}

\newcommand{\booleanContext}{\scriptsize
	\begin{tikzpicture}
		\draw(0,0) node{
			\begin{tabular}{|c||c|c|c|}
				\hline
				$\neq$ & $a$ & $b$ & $c$ \\
				\hline\hline
				$a$ & & $\times$ & $\times$ \\
				\hline
				$b$ & $\times$ & & $\times$ \\
				\hline
				$c$ & $\times$ & $\times$& \\
				\hline
			\end{tabular}};
		\draw(4,0) node{
			\begin{tikzpicture}
				\def\r{.8};
				\draw(2,1) node[circle,draw,scale=\r](n0){};
				\draw(1,2) node[circle,draw,scale=\r,label=below left:$a$](n1){};
				\draw(2,2) node[circle,draw,scale=\r,label=below left:$b$](n2){};
				\draw(3,2) node[circle,draw,scale=\r,label=below right:$c$](n3){};
				\draw(1,3) node[circle,draw,scale=\r,label=above left:$c$](n4){};
				\draw(2,3) node[circle,draw,scale=\r,label=above right:$b$](n5){};
				\draw(3,3) node[circle,draw,scale=\r,label=above right:$a$](n6){};
				\draw(2,4) node[circle,draw,scale=\r](n7){};
				\draw (n0) -- (n1) -- (n4) -- (n7);
				\draw (n0) -- (n2) -- (n6) -- (n7);
				\draw (n0) -- (n3) -- (n5) -- (n7);
				\draw (n1) -- (n5);
				\draw (n2) -- (n4);
				\draw (n3) -- (n6);
			\end{tikzpicture}};
	\end{tikzpicture}
}

\newcommand{\contraordinalChain}{\scriptsize
	\begin{tikzpicture}
		\draw(0,0) node{
			\begin{tabular}{|c||c|c|c|c|}
				\hline
				$\not\geq$ & $a$ & $b$ & $c$ & $d$\\
				\hline\hline
				$a$ & & $\times$ & $\times$ & $\times$\\
				\hline
				$b$ & & & $\times$ & $\times$\\
				\hline
				$c$ & & & & $\times$\\
				\hline
				$d$ & & & & \\
				\hline
			\end{tabular}};
		\draw(4,0) node{
			\begin{tikzpicture}\scriptsize
				\def\r{.8};
				\draw(1,1) node[circle,draw,scale=\r,label=above left:$a$](n0){};
				\draw(1,2) node[circle,draw,scale=\r,label=below right:$a$,
				  label=above left:$b$](n1){};
				\draw(1,3) node[circle,draw,scale=\r,label=below right:$b$,
				  label=above left:$c$](n2){};
				\draw(1,4) node[circle,draw,scale=\r,label=below right:$c$,
				  label=above left:$d$](n3){};
				\draw(1,5) node[circle,draw,scale=\r,label=below right:$d$](n4){};
				\draw (n0) -- (n1) -- (n2) -- (n3) -- (n4);
			\end{tikzpicture}};
	\end{tikzpicture}
}

\newcommand{\exampleOnePartition}{\scriptsize
	\begin{tikzpicture}[scale=.75]
		\draw[step=1,gray!50!white,very thin](0,1) grid (6,6);
		\foreach \x in {.5,1.5,2.5,3.5} {\foreach \y in {3.5,4.5,5.5} {\draw(\x,\y) node{$2$};}}
		\draw(.5,2.5) node{$2$};
		\draw(4.5,5.5) node{$2$};
		\foreach \x in {.5,1.5} {\foreach \y in {1.5} {\draw(\x,\y) node{$1$};}}
		\draw(1.5,2.5) node{$1$};
		\foreach \x in {4.5} {\foreach \y in {3.5,4.5} {\draw(\x,\y) node{$1$};}}
		\draw(5.5,5.5) node{$1$};
		\foreach \x in {2.5,3.5,4.5} {\foreach \y in {1.5,2.5} {\draw(\x,\y) node{$0$};}}
		\foreach \x in {5.5} {\foreach \y in {1.5,2.5,3.5,4.5} {\draw(\x,\y) node{$0$};}}
	\end{tikzpicture}
}

\newcommand{\exampleOneRTable}{\scriptsize
	\begin{tabular}{|c||c|c|c|c|c|c|}
		\hline
		$R$ & $b_{1}$ & $b_{2}$ & $b_{3}$ & $b_{4}$ & $b_{5}$ & $b_{6}$ \\
		\hline\hline
		$a_{1}$ & & $\times$ & $\times$ & $\times$ & $\times$ & $\times$ \\
		\hline
		$a_{2}$ & & & $\times$ & $\times$ & $\times$ & $\times$ \\
		\hline
		$a_{3}$ & & & $\times$ & $\times$ & $\times$ & $\times$ \\
		\hline
		$a_{4}$ & & & & & & $\times$ \\
		\hline
		$a_{5}$ & & & & & & \\
		\hline
	\end{tabular}
}

\newcommand{\exampleOneSTable}{\scriptsize
	\begin{tabular}{|c||c|c|c|c|c|}
		\hline
		$S$ & $a_{1}$ & $a_{2}$ & $a_{3}$ & $a_{4}$ & $a_{5}$ \\
		\hline\hline
		$b_{1}$ & & $\times$ & $\times$ & $\times$ & $\times$ \\
		\hline
		$b_{2}$ & & & & $\times$ & $\times$ \\
		\hline
		$b_{3}$ & & & & $\times$ & $\times$ \\
		\hline
		$b_{4}$ & & & & $\times$ & $\times$ \\
		\hline
		$b_{5}$ & & & & & \\
		\hline
		$b_{6}$ & & & & & \\
		\hline
	\end{tabular}
}

\newcommand{\exampleOnePoset}{
	\begin{tikzpicture}
		\def\r{.8};
		\draw(1,0) node[draw,circle,scale=\r](a1){};
		\draw(2,0) node[fill=black,circle,scale=\r](b1){};
		\draw(1,1) node[draw,circle,scale=\r](a2){};
		\draw(2,1) node[fill=black,circle,scale=\r](b2){};
		\draw(1,2) node[draw,circle,scale=\r](a3){};
		\draw(1.5,3) node[fill=black,circle,scale=\r](b3){};
		\draw(1.5,4) node[fill=black,circle,scale=\r](b4){};
		\draw(1,5) node[fill=black,circle,scale=\r](b5){};
		\draw(2,5) node[draw,circle,scale=\r](a4){};
		\draw(2,6) node[draw,circle,scale=\r](a5){};
		\draw(1,6) node[fill=black,circle,scale=\r](b6){};
		\draw(a1) -- (a2) -- (a3) -- (b3) -- (b4) -- (a4) -- (a5);
		\draw(a1) -- (b2);
		\draw(b1) -- (a2);
		\draw(b1) -- (b2) -- (b3);
		\draw(a4) -- (b6);
		\draw(b4) -- (b5) -- (b6);
	\end{tikzpicture}
}

\newcommand{\exampleThreePartition}{\scriptsize
	\begin{tikzpicture}[scale=.75]
		\draw(-.5,0) node{};
		\draw(4.5,0) node{};
		\draw[step=1,gray!50!white,very thin](0,0) grid (4,6);
		\foreach \x in {.5,1.5,2.5,3.5} {\foreach \y in {5.5} {\draw(\x,\y) node{$1$};}}
		\foreach \x in {.5,1.5,2.5} {\foreach \y in {3.5,4.5} {\draw(\x,\y) node{$1$};}}
		\foreach \x in {3.5} {\foreach \y in {3.5,4.5} {\draw(\x,\y) node{$0$};}}
		\foreach \x in {.5} {\foreach \y in {1.5,2.5} {\draw(\x,\y) node{$1$};}}
		\foreach \x in {1.5,2.5,3.5} {\foreach \y in {1.5,2.5} {\draw(\x,\y) node{$0$};}}
		\foreach \x in {.5,1.5,2.5,3.5} {\foreach \y in {.5} {\draw(\x,\y) node{$0$};}}
	\end{tikzpicture}
}

\newcommand{\exampleThreeRTable}{\scriptsize
	\begin{tabular}{|c||c|c|c|c|}
		      \hline
		      $R$ & $b_{1}$ & $b_{2}$ & $b_{3}$ & $b_{4}$ \\
		      \hline\hline
		      $a_{1}$ & & & & \\
		      \hline
		      $a_{2}$ & & & & \\
		      \hline
		      $a_{3}$ & & & & \\
		      \hline
		      $a_{4}$ & & & & \\
		      \hline
		      $a_{5}$ & & & & \\
		      \hline
		      $a_{6}$ & & & & \\
		      \hline
		\end{tabular}
}

\newcommand{\exampleThreeSTable}{\scriptsize
	\begin{tabular}{|c||c|c|c|c|c|c|}
		      \hline
		      $S$ & $a_{1}$ & $a_{2}$ & $a_{3}$ & $a_{4}$ & $a_{5}$ & $a_{6}$ \\
		      \hline\hline
		      $b_{1}$ & & $\times$ & $\times$ & $\times$ & $\times$ & $\times$ \\
		      \hline
		      $b_{2}$ & & & & $\times$ & $\times$ & $\times$ \\
		      \hline
		      $b_{3}$ & & & & $\times$ & $\times$ & $\times$ \\
		      \hline
		      $b_{4}$ & & & & & & $\times$ \\
		      \hline
		\end{tabular}
}

\newcommand{\exampleThreePoset}{\scriptsize
	\begin{tikzpicture}
		\def\x{1};
		\def\y{1};
		\draw(0*\x,1*\y) node{};
		\draw(3*\x,1*\y) node{};
		\draw(1*\x,1*\y) node[draw,circle,scale=.75](a1){};
		\draw(1*\x,2*\y) node[draw,circle,scale=.75](a2){};
		\draw(1*\x,3*\y) node[draw,circle,scale=.75](a3){};
		\draw(1*\x,4*\y) node[draw,circle,scale=.75](a4){};
		\draw(1*\x,5*\y) node[draw,circle,scale=.75](a5){};
		\draw(1*\x,6*\y) node[draw,circle,scale=.75](a6){};
		\draw(2*\x,1*\y) node[draw,circle,fill=black,scale=.75](b1){};
		\draw(2*\x,2*\y) node[draw,circle,fill=black,scale=.75](b2){};
		\draw(2*\x,3*\y) node[draw,circle,fill=black,scale=.75](b3){};
		\draw(2*\x,5*\y) node[draw,circle,fill=black,scale=.75](b4){};
		\draw(a1) -- (a2) -- (a3) -- (a4) -- (a5) -- (a6);
		\draw(b1) -- (b2) -- (b3) -- (b4);
		\draw(b1) -- (a2);
		\draw(b3) -- (a4);
		\draw(b4) -- (a6);
	\end{tikzpicture}
}

\newcommand{\exampleThreeTTable}{\scriptsize
	\begin{tabular}{|c||c|c|c|c|c|c|}
		      \hline
		      $T$ & $a_{1}$ & $a_{2}$ & $a_{3}$ & $a_{4}$ & $a_{5}$ & $a_{6}$ \\
		      \hline\hline
		      $b_{1}$ & $\times$ & $\times$ & $\times$ & $\times$ & $\times$ & \\
		      \hline
		      $b_{2}$ & $\times$ & $\times$ & $\times$ & & & \\
		      \hline
		      $b_{3}$ & $\times$ & $\times$ & $\times$ & & & \\
		      \hline
		      $b_{4}$ & $\times$ & & & & & \\
		      \hline
		\end{tabular}
}

\newcommand{\exampleThreeGalois}{\scriptsize
	\begin{tikzpicture}[>=stealth']
		\def\x{1};
		\def\y{1};
		\draw(1*\x,1*\y) node[draw,circle,fill=black,scale=.75](b11){};
		\draw(1*\x,2*\y) node[draw,circle,fill=black,scale=.75](b12){};
		\draw(1*\x,3*\y) node[draw,circle,fill=black,scale=.75](b13){};
		\draw(1*\x,4*\y) node[draw,circle,fill=black,scale=.75](b14){};
		\draw(1*\x,5*\y) node[draw,circle,fill=black,scale=.75](b15){};
		\draw(b11) -- (b12) -- (b13) -- (b14) -- (b15);
		\draw(4*\x,1*\y) node[draw,circle,scale=.75](a11){};
		\draw(4*\x,2*\y) node[draw,circle,scale=.75](a12){};
		\draw(4*\x,3*\y) node[draw,circle,scale=.75](a13){};
		\draw(4*\x,4*\y) node[draw,circle,scale=.75](a14){};
		\draw(4*\x,5*\y) node[draw,circle,scale=.75](a15){};
		\draw(4*\x,6*\y) node[draw,circle,scale=.75](a16){};
		\draw(4*\x,7*\y) node[draw,circle,scale=.75](a17){};
		\draw(a11) -- (a12) -- (a13) -- (a14) -- (a15) -- (a16) -- (a17);
		\draw[dashed,->](b11) -- (a17);
		\draw[dashed,->](b12) -- (a16);
		\draw[dashed,->](b13) -- (a14);
		\draw[dashed,->](b14) -- (a14);
		\draw[dashed,->](b15) -- (a12);
		\draw(2.5*\x,1*\y) node{\normalsize $\psi$};
		\draw(6*\x,1*\y) node[draw,circle,scale=.75](a21){};
		\draw(6*\x,2*\y) node[draw,circle,scale=.75](a22){};
		\draw(6*\x,3*\y) node[draw,circle,scale=.75](a23){};
		\draw(6*\x,4*\y) node[draw,circle,scale=.75](a24){};
		\draw(6*\x,5*\y) node[draw,circle,scale=.75](a25){};
		\draw(6*\x,6*\y) node[draw,circle,scale=.75](a26){};
		\draw(6*\x,7*\y) node[draw,circle,scale=.75](a27){};
		\draw(a21) -- (a22) -- (a23) -- (a24) -- (a25) -- (a26) -- (a27);
		\draw(9*\x,1*\y) node[draw,circle,fill=black,scale=.75](b21){};
		\draw(9*\x,2*\y) node[draw,circle,fill=black,scale=.75](b22){};
		\draw(9*\x,3*\y) node[draw,circle,fill=black,scale=.75](b23){};
		\draw(9*\x,4*\y) node[draw,circle,fill=black,scale=.75](b24){};
		\draw(9*\x,5*\y) node[draw,circle,fill=black,scale=.75](b25){};
		\draw(b21) -- (b22) -- (b23) -- (b24) -- (b25);
		\draw[dashed,->](a21) -- (b25);
		\draw[dashed,->](a22) -- (b25);
		\draw[dashed,->](a23) -- (b24);
		\draw[dashed,->](a24) -- (b24);
		\draw[dashed,->](a25) -- (b22);
		\draw[dashed,->](a26) -- (b22);
		\draw[dashed,->](a27) -- (b21);
		\draw(7.5*\x,1*\y) node{\normalsize $\varphi$};
	\end{tikzpicture}
}

\newcommand{\connectedComponentsNormal}{\scriptsize
	\begin{tikzpicture}
		\draw(0,0) node[fill=black,circle,scale=.75](v1){};
		\draw(1,0) node[fill=black,circle,scale=.75](v2){};
		\draw(2,0) node[fill=black,circle,scale=.75](v3){};
		\draw(3,0) node[fill=black,circle,scale=.75](v4){};
		\draw(4,0) node[fill=black,circle,scale=.75](v5){};
		\draw(5,0) node[fill=black,circle,scale=.75](v6){};
		\draw(6,0) node[fill=black,circle,scale=.75](v7){};
		\draw(1,1) node[fill=green!50!black,circle,scale=.75](p1){};
		\draw(2,1) node[fill=green!50!black,circle,scale=.75](p2){};
		\draw(4,1) node[fill=green!50!black,circle,scale=.75](p3){};
		\draw(5,1) node[fill=green!50!black,circle,scale=.75](p4){};
		\draw(7,0) node[fill=green!50!black,circle,scale=.75](p5){};
		\draw(v1) -- (p1) -- (v2) -- (p2) -- (v1);
		\draw(v3) -- (p1) -- (v4) -- (p2) -- (v3);
		\draw(v5) -- (p3) -- (v6);
		\draw(p4) -- (v5);
	\end{tikzpicture}
}

\newcommand{\connectedComponentsFlipped}{\scriptsize
	\begin{tikzpicture}
		\draw(0,0) node[fill=black,circle,scale=.75](v1){};
		\draw(1,0) node[fill=black,circle,scale=.75](v2){};
		\draw(2,0) node[fill=black,circle,scale=.75](v3){};
		\draw(3,0) node[fill=black,circle,scale=.75](v4){};
		\draw(4,1) node[fill=black,circle,scale=.75](v5){};
		\draw(5,1) node[fill=black,circle,scale=.75](v6){};
		\draw(6,0) node[fill=black,circle,scale=.75](v7){};
		\draw(1,1) node[fill=green!50!black,circle,scale=.75](p1){};
		\draw(2,1) node[fill=green!50!black,circle,scale=.75](p2){};
		\draw(4,0) node[fill=green!50!black,circle,scale=.75](p3){};
		\draw(5,0) node[fill=green!50!black,circle,scale=.75](p4){};
		\draw(7,0) node[fill=green!50!black,circle,scale=.75](p5){};
		\draw(v1) -- (p1) -- (v2) -- (p2) -- (v1);
		\draw(v3) -- (p1) -- (v4) -- (p2) -- (v3);
		\draw(v5) -- (p3) -- (v6);
		\draw(p4) -- (v5);
	\end{tikzpicture}
}

\newcommand{\exampleFourColoring}{\scriptsize
	\begin{tikzpicture}[>=stealth']
		\draw(0,0) node[draw,circle,scale=.8,
		  label=below:{\color{gray!50!black}$v_{1}^{(1)}$}](v1){$1$};
		\draw(2,0) node[draw,circle,scale=.8,
		  label=below:{\color{gray!50!black}$v_{2}^{(1)}$}](v2){$1$};
		\draw(4,0) node[draw,circle,scale=.8,
		  label=below:{\color{gray!50!black}$v_{3}^{(1)}$}](v3){$2$};
		\draw(6,0) node[draw,circle,scale=.8,
		  label=below:{\color{gray!50!black}$v_{4}^{(1)}$}](v4){$2$};
		\draw(0,2) node[draw,circle,scale=.8,
		  label=above:{\color{gray!50!black}$v_{1}^{(2)}$}](p1){$4$};
		\draw(2,2) node[draw,circle,scale=.8,
		  label=above:{\color{gray!50!black}$v_{2}^{(2)}$}](p2){$3$};
		\draw(4,2) node[draw,circle,scale=.8,
		  label=above:{\color{gray!50!black}$v_{3}^{(2)}$}](p3){$2$};
		\draw(6,2) node[draw,circle,scale=.8,
		  label=above:{\color{gray!50!black}$v_{4}^{(2)}$}](p4){$4$};
		\draw[->](v1) -- (p1);
		\draw[->](v1) -- (p2);
		\draw[->](v1) -- (p3);
		\draw[->](v1) -- (p4);
		\draw[->](v2) -- (p1);
		\draw[->](v2) -- (p2);
		\draw[->](v2) -- (p3);
		\draw[->](v2) -- (p4);
		\draw[->](v3) -- (p1);
		\draw[->](v3) -- (p2);
		\draw[->](v3) -- (p3);
		\draw[->](v3) -- (p4);
		\draw[->](v4) -- (p1);
		\draw[->](v4) -- (p2);
		\draw[->](v4) -- (p3);
		\draw[->](v4) -- (p4);
	\end{tikzpicture}
}

\newcommand{\exampleFourRTable}{\scriptsize
	\begin{tabular}{|c||c|c|c|}
		\hline
		$R$ & $c_{1}$ & $c_{2}$ & $c_{3}$ \\
		\hline\hline
		$a_{1}$ & & & \\
		\hline
		$a_{2}$ & & & \\
		\hline
		$a_{3}$ & & & $\times$ \\
		\hline
		$a_{4}$ & & & $\times$ \\
		\hline
	\end{tabular}
}

\newcommand{\exampleFourSTable}{\scriptsize
	\begin{tabular}{|c||c|c|c|c|}
		\hline
		$S$ & $a_{1}$ & $a_{2}$ & $a_{3}$ & $a_{4}$ \\
		\hline\hline
		$c_{1}$ & & $\times$ & $\times$ & \\
		\hline
		$c_{2}$ & & & $\times$ & \\
		\hline
		$c_{3}$ & & & & \\
		\hline
	\end{tabular}
}

\newcommand{\exampleFourPoset}{\scriptsize
	\begin{tikzpicture}
		\draw(3,1) node[circle,fill=black,scale=.75](c1){};
		\draw(3,2) node[circle,fill=black,scale=.75](c2){};
		\draw(3,4) node[circle,fill=black,scale=.75](c3){};
		\draw(1,3) node[circle,fill=green!50!black,scale=.75,label=below left:$a_{1}$](a1){};
		\draw(2,3) node[circle,fill=green!50!black,scale=.75,label=below left:$a_{2}$](a2){};
		\draw(4,3) node[circle,fill=green!50!black,scale=.75,label=below right:$a_{3}$](a3){};
		\draw(5,3) node[circle,fill=green!50!black,scale=.75,label=below right:$a_{4}$](a4){};
		\draw(c1) -- (c2);
		\draw(a3) -- (c3);
		\draw(a4) -- (c3);
		\draw(c1) -- (a2);
		\draw(c2) -- (a3);
	\end{tikzpicture}
}

\newcommand{\exampleFiveSolidPartition}[3]{
	\begin{tikzpicture}\scriptsize
		\draw[gray!50!white](0,0) -- (0,1) -- (1,1) -- (1,0) -- cycle;
		\draw[gray!50!white](0,1) -- (-.62,.38) -- (-.62,-.62) -- (0,0);
		\draw[gray!50!white](-.62,-.62) -- (.38,-.62) -- (1,0);
		\draw(0,.5) node[fill=white]{$c$};
		\draw(.5,0) node[fill=white]{$b$};
		\draw(-.31,-.31) node[fill=white]{$a$};
		\draw(.5,.5) node{#1};
		\draw(.15,-.31) node{#2};
		\draw(-.31,.15) node{#3};
	\end{tikzpicture}
}

\newcommand{\exampleFiveSmallPartition}[3]{
	\begin{tikzpicture}\scriptsize
		\draw[gray!50!white](0,0) -- (0,1) -- (1,1) -- (1,0) -- cycle;
		\draw(-.2,.5) node{#1};
		\draw(.5,1.2) node{#2};
		\draw(.5,.5) node{#3};
	\end{tikzpicture}
}

\newcommand{\exampleFivePosetOne}{
	\begin{tikzpicture}\scriptsize
		\draw(0,0) node[circle,draw,scale=.8,label=right:$c$](v1){};
		\draw(0,1) node[circle,draw,scale=.8,label=right:$b$](v2){};
		\draw(0,2) node[circle,draw,scale=.8,label=right:$a$](v3){};
		\draw(v1) -- (v2) -- (v3);
	\end{tikzpicture}
}
\newcommand{\exampleFivePosetTwo}{
	\begin{tikzpicture}\scriptsize
		\draw(0,0) node[circle,draw,scale=.8,label=right:${a,b,c}$]{};
	\end{tikzpicture}
}

\newcommand{\nodeaa}[2]{
  \draw[black](#1,#2) arc (180:360:.1cm);
  \draw[black](#1,#2) arc(180:0:.1cm);}
\newcommand{\nodeab}[2]{
  \draw[black](#1,#2) arc (180:360:.1cm);
  \draw[black](#1,#2) arc(180:0:.1cm);}
\newcommand{\nodeba}[2]{
  \filldraw[black](#1,#2) arc (180:360:.1cm) -- cycle;
  \filldraw[black](#1,#2) arc(180:0:.1cm);}
\newcommand{\nodebb}[2]{
  \filldraw[black](#1,#2) arc (180:360:.1cm);
  \filldraw[black](#1,#2) arc(180:0:.1cm) -- cycle;}

\newcommand{\lOne}[2]{
	\begin{tikzpicture}[rotate=#1,scale=#2]
		\nodeba{.5}{0};
		\nodebb{.5}{.5};
		\nodeaa{.5}{1};
		\nodeab{.5}{1.5};
		\draw(.6,.1) -- (.6,.4);
		\draw(.6,.6) -- (.6,.9);
		\draw(.6,1.1) -- (.6,1.4);
	\end{tikzpicture}
}

\newcommand{\lTwo}[2]{
	\begin{tikzpicture}[rotate=#1,scale=#2]
		\nodeba{.5}{0};
		\nodebb{1}{.5};
		\nodeaa{0}{.5};
		\nodeab{.5}{1};
		\draw(.53,.08) -- (.17,.43);
		\draw(.67,.08) -- (1.03,.43);
		\draw(.17,.58) -- (.53,.93);
		\draw(1.03,.58) -- (.67,.93);
	\end{tikzpicture}
}

\newcommand{\lThree}[2]{
	\begin{tikzpicture}[rotate=#1,scale=#2]
		\nodeba{1}{0};
		\nodebb{1}{.5};
		\nodeaa{0}{.5};
		\nodeab{.5}{1};
		\draw(1.1,.1) -- (1.1,.4);
		\draw(.17,.58) -- (.53,.93);
		\draw(1.03,.58) -- (.67,.93);
	\end{tikzpicture}
}

\newcommand{\lFour}[2]{
	\begin{tikzpicture}[rotate=#1,scale=#2]
		\nodeba{.5}{0};
		\nodebb{1}{.5};
		\nodeaa{0}{.5};
		\nodeab{0}{1};
		\draw(.53,.08) -- (.17,.43);
		\draw(.67,.08) -- (1.03,.43);
		\draw(.1,.6) -- (.1,.9);

	\end{tikzpicture}
}

\newcommand{\lFive}[2]{
	\begin{tikzpicture}[rotate=#1,scale=#2]
		\nodeba{.5}{0};
		\nodeaa{.5}{.5};
		\nodebb{.5}{1};
		\nodeab{.5}{1.5};
		\draw(.6,.1) -- (.6,.4);
		\draw(.6,.6) -- (.6,.9);
		\draw(.6,1.1) -- (.6,1.4);
	\end{tikzpicture}
}

\newcommand{\lSix}[2]{
	\begin{tikzpicture}[rotate=#1,scale=#2]
		\nodeaa{0}{0};
		\nodeab{0}{.5};
		\nodeba{.5}{0};
		\nodebb{.5}{.5};
		\draw(.1,.1) -- (.1,.4);
		\draw(.6,.1) -- (.6,.4);
		\draw(.53,.08) -- (.17,.43);
	\end{tikzpicture}
}

\newcommand{\lSeven}[2]{
	\begin{tikzpicture}[rotate=#1,scale=#2]
		\nodeaa{0}{0};
		\nodeab{.5}{1};
		\nodeba{.5}{0};
		\nodebb{.5}{.5};
		\draw(.6,.1) -- (.6,.4);
		\draw(.6,.6) -- (.6,.9);
		\draw(.17,.08) -- (.53,.43);
	\end{tikzpicture}
}

\newcommand{\lEight}[2]{
	\begin{tikzpicture}[rotate=#1,scale=#2]
		\nodeba{0}{0};
		\nodeaa{0}{.5};
		\nodeab{0}{1};
		\nodebb{.5}{1};
		\draw(.1,.1) -- (.1,.4);
		\draw(.1,.6) -- (.1,.9);
		\draw(.17,.58) -- (.53,.93);
	\end{tikzpicture}
}

\newcommand{\lNine}[2]{
	\begin{tikzpicture}[rotate=#1,scale=#2]
		\nodeaa{0}{0};
		\nodeab{0}{.5};
		\nodeba{.5}{0};
		\nodebb{.5}{.5};
		\draw(.1,.1) -- (.1,.4);
		\draw(.6,.1) -- (.6,.4);
	\end{tikzpicture}
}

\newcommand{\lTen}[2]{
	\begin{tikzpicture}[rotate=#1,scale=#2]
		\nodeaa{0}{0};
		\nodeab{0}{.5};
		\nodeba{.5}{0};
		\nodebb{.5}{.5};
		\draw(.1,.1) -- (.1,.4);
		\draw(.6,.1) -- (.6,.4);
		\draw(.17,.08) -- (.53,.43);
		\draw(.53,.08) -- (.17,.43);
	\end{tikzpicture}
}

\newcommand{\lEleven}[2]{
	\begin{tikzpicture}[rotate=#1,scale=#2]
		\nodeaa{.5}{0};
		\nodeba{.5}{.5};
		\nodebb{.5}{1};
		\nodeab{.5}{1.5};
		\draw(.6,.1) -- (.6,.4);
		\draw(.6,.6) -- (.6,.9);
		\draw(.6,1.1) -- (.6,1.4);
	\end{tikzpicture}
}

\newcommand{\lTwelve}[2]{
	\begin{tikzpicture}[rotate=#1,scale=#2]
		\nodeba{.5}{0};
		\nodeaa{.5}{.5};
		\nodeab{.5}{1};
		\nodebb{.5}{1.5};
		\draw(.6,.1) -- (.6,.4);
		\draw(.6,.6) -- (.6,.9);
		\draw(.6,1.1) -- (.6,1.4);
	\end{tikzpicture}
}

\newcommand{\plp}[4]{\scriptsize
	\begin{tikzpicture}[scale=.66]
		\draw[step=1,gray!50!white,very thin](0,0) grid (2,2);
		\draw(.5,1.5) node{#1};
		\draw(1.5,1.5) node{#2};
		\draw(.5,.5) node{#3};
		\draw(1.5,.5) node{#4};
		\draw(1,2.05) node{};
	 \end{tikzpicture}}

\newcommand{\rTabOne}{\scriptsize
	\begin{tikzpicture}
		\draw(1,1) node{
			\begin{tabular}{|c||c|c|}
				  \hline
				$R$ & $\bar{c}_{1}$ & $\bar{c}_{2}$ \\
				  \hline\hline
				$c_{1}$ & & \\
				  \hline
				$c_{2}$ & & \\
				  \hline
			\end{tabular}};
	\end{tikzpicture}}

\newcommand{\rTabTwo}{\scriptsize
	\begin{tikzpicture}
		\draw(1,1) node{
			\begin{tabular}{|c||c|c|}
				  \hline
				$R$ & $\bar{c}_{1}$ & $\bar{c}_{2}$ \\
				  \hline\hline
				$c_{1}$ & & $\times$ \\
				  \hline
				$c_{2}$ & & \\
				  \hline
			\end{tabular}};
	\end{tikzpicture}}

\newcommand{\rTabThree}{\scriptsize
	\begin{tikzpicture}
		\draw(1,1) node{
			\begin{tabular}{|c||c|c|}
				  \hline
				$R$ & $\bar{c}_{1}$ & $\bar{c}_{2}$ \\
				  \hline\hline
				$c_{1}$ & & $\times$ \\
				  \hline
				$c_{2}$ & & $\times$ \\
				  \hline
			\end{tabular}};
	\end{tikzpicture}}

\newcommand{\rTabFour}{\scriptsize
	\begin{tikzpicture}
		\draw(1,1) node{
			\begin{tabular}{|c||c|c|}
				  \hline
				$R$ & $\bar{c}_{1}$ & $\bar{c}_{2}$ \\
				  \hline\hline
				$c_{1}$ & $\times$ & $\times$ \\
				  \hline
				$c_{2}$ & & \\
				  \hline
			\end{tabular}};
	\end{tikzpicture}}

\newcommand{\rTabFive}{\scriptsize
	\begin{tikzpicture}
		\draw(1,1) node{
			\begin{tabular}{|c||c|c|}
				  \hline
				$R$ & $\bar{c}_{1}$ & $\bar{c}_{2}$ \\
				  \hline\hline
				$c_{1}$ & $\times$ & $\times$ \\
				  \hline
				$c_{2}$ & & $\times$ \\
				  \hline
			\end{tabular}};
	\end{tikzpicture}}

\newcommand{\rTabSix}{\scriptsize
	\begin{tikzpicture}
		\draw(1,1) node{
			\begin{tabular}{|c||c|c|}
				  \hline
				$R$ & $\bar{c}_{1}$ & $\bar{c}_{2}$ \\
				  \hline\hline
				$c_{1}$ & $\times$ & $\times$ \\
				  \hline
				$c_{2}$ & $\times$ & $\times$ \\
				  \hline
			\end{tabular}};
	\end{tikzpicture}}

\newcommand{\sTabOne}{\scriptsize
	\begin{tikzpicture}
		\draw(1,1) node{
			\begin{tabular}{|c||c|c|}
				  \hline
				$S$ & $c_{1}$ & $c_{2}$ \\
				  \hline\hline
				$\bar{c}_{1}$ & $\times$ & $\times$ \\
				  \hline
				$\bar{c}_{2}$ & $\times$ & $\times$ \\
				  \hline
			\end{tabular}};
	\end{tikzpicture}}

\newcommand{\sTabTwo}{\scriptsize
	\begin{tikzpicture}
		\draw(1,1) node{
			\begin{tabular}{|c||c|c|}
				  \hline
				$S$ & $c_{1}$ & $c_{2}$ \\
				  \hline\hline
				$\bar{c}_{1}$ & $\times$ & $\times$ \\
				  \hline
				$\bar{c}_{2}$ & & $\times$ \\
				  \hline
			\end{tabular}};
	\end{tikzpicture}}

\newcommand{\sTabThree}{\scriptsize
	\begin{tikzpicture}
		\draw(1,1) node{
			\begin{tabular}{|c||c|c|}
				  \hline
				$S$ & $c_{1}$ & $c_{2}$ \\
				  \hline\hline
				$\bar{c}_{1}$ & $\times$ & $\times$ \\
				  \hline
				$\bar{c}_{2}$ & & \\
				  \hline
			\end{tabular}};
	\end{tikzpicture}}

\newcommand{\sTabFour}{\scriptsize
	\begin{tikzpicture}
		\draw(1,1) node{
			\begin{tabular}{|c||c|c|}
				  \hline
				$S$ & $c_{1}$ & $c_{2}$ \\
				  \hline\hline
				$\bar{c}_{1}$ & & $\times$ \\
				  \hline
				$\bar{c}_{2}$ & & $\times$ \\
				  \hline
			\end{tabular}};
	\end{tikzpicture}}

\newcommand{\sTabFive}{\scriptsize
	\begin{tikzpicture}
		\draw(1,1) node{
			\begin{tabular}{|c||c|c|}
				  \hline
				$S$ & $c_{1}$ & $c_{2}$ \\
				  \hline\hline
				$\bar{c}_{1}$ & & $\times$ \\
				  \hline
				$\bar{c}_{2}$ & & \\
				  \hline
			\end{tabular}};
	\end{tikzpicture}}

\newcommand{\sTabSix}{\scriptsize
	\begin{tikzpicture}
		\draw(1,1) node{
			\begin{tabular}{|c||c|c|}
				  \hline
				$S$ & $c_{1}$ & $c_{2}$ \\
				  \hline\hline
				$\bar{c}_{1}$ & & \\
				  \hline
				$\bar{c}_{2}$ & & \\
				  \hline
			\end{tabular}};
	\end{tikzpicture}}

\newcommand{\nodega}[2]{
  \filldraw[black](#1,#2) arc (180:360:.1cm);
  \filldraw[black](#1,#2) arc(180:0:.1cm);}
\newcommand{\nodegb}[2]{
  \filldraw[green!67!black](#1,#2) arc (180:360:.1cm) -- cycle;
  \filldraw[green!67!black](#1,#2) arc(180:0:.1cm);}
\newcommand{\nodegc}[2]{
  \filldraw[draw=black,fill=white](#1,#2) arc (180:360:.1cm);
  \filldraw[draw=black,fill=white](#1,#2) arc(180:0:.1cm);}

\newcommand{\uOne}[2]{
	\begin{tikzpicture}[rotate=#1,scale=#2]
		\nodega{0}{0};
		\nodegb{0}{.5};
		\nodegc{0}{1};
		\draw(.1,.1) -- (.1,.4);
		\draw(.1,.6) -- (.1,.9);
	\end{tikzpicture}
}
\newcommand{\uTwo}[2]{
	\begin{tikzpicture}[rotate=#1,scale=#2]
		\nodega{0}{0};
		\nodegb{1}{0};
		\nodegc{.5}{.5};
		\draw(.17,.08) -- (.53,.43);
		\draw(1.03,.08) -- (.67,.43);
	\end{tikzpicture}
}
\newcommand{\uThree}[2]{
	\begin{tikzpicture}[rotate=#1,scale=#2]
		\nodegb{0}{.5};
		\nodega{.5}{0};
		\nodegc{1}{.5};
		\draw(.17,.43) -- (.53,.08);
		\draw(1.03,.43) -- (.67,.08);
	\end{tikzpicture}
}
\newcommand{\uFour}[2]{
	\begin{tikzpicture}[rotate=#1,scale=#2]
		\nodegb{0}{0};
		\nodega{0}{.5};
		\nodegc{0}{1};
		\draw(.1,.1) -- (.1,.4);
		\draw(.1,.6) -- (.1,.9);
	\end{tikzpicture}
}
\newcommand{\uFive}[2]{
	\begin{tikzpicture}[rotate=#1,scale=#2]
		\nodegb{0}{0};
		\nodegc{0}{.5};
		\nodega{.5}{.25};
		\draw(.1,.1) -- (.1,.4);
	\end{tikzpicture}
}
\newcommand{\uSix}[2]{
	\begin{tikzpicture}[rotate=#1,scale=#2]
		\nodega{0}{0};
		\nodegc{0}{.5};
		\nodegb{.5}{.25};
		\draw(.1,.1) -- (.1,.4);
	\end{tikzpicture}
}
\newcommand{\uSeven}[2]{
	\begin{tikzpicture}[rotate=#1,scale=#2]
		\nodega{.5}{0};
		\nodegb{.5}{.5};
		\nodegc{0}{.25};
		\draw(.6,.1) -- (.6,.4);
	\end{tikzpicture}
}
\newcommand{\uEight}[2]{
	\begin{tikzpicture}[rotate=#1,scale=#2]
		\nodega{0}{0};
		\nodegc{0}{.5};
		\nodegb{0}{1};
		\draw(.1,.1) -- (.1,.4);
		\draw(.1,.6) -- (.1,.9);
	\end{tikzpicture}
}
\newcommand{\uNine}[2]{
	\begin{tikzpicture}[rotate=#1,scale=#2]
		\nodega{0}{.5};
		\nodegb{.5}{0};
		\nodegc{1}{.5};
		\draw(.17,.43) -- (.53,.08);
		\draw(1.03,.43) -- (.67,.08);
	\end{tikzpicture}
}
\newcommand{\uTen}[2]{
	\begin{tikzpicture}[rotate=#1,scale=#2]
		\nodega{0}{0};
		\nodegb{.5}{0};
		\nodegc{1}{0};
	\end{tikzpicture}
}

\newcommand{\solPar}[3]{
	\begin{tikzpicture}\scriptsize
		\draw(0,1.25) node{};
		\draw[gray!50!white](0,0) -- (0,1) -- (1,1) -- (1,0) -- cycle;
		\draw[gray!50!white](0,1) -- (-.62,.38) -- (-.62,-.62) -- (0,0);
		\draw[gray!50!white](-.62,-.62) -- (.38,-.62) -- (1,0);
		\nodegc{-.1}{.5};
		\nodegb{.4}{0};
		\nodega{-.31}{-.21};
		\draw(.15,-.31) node{#1};
		\draw(.5,.5) node{#2};
		\draw(-.31,.15) node{#3};
	\end{tikzpicture}
}

\newcommand{\planeOne}[1]{
	\begin{tikzpicture}\scriptsize
		\draw[gray!50!white](0,0) -- (0,1) -- (1,1) -- (1,0) -- cycle;
		\nodega{-.25}{.5};
		\nodegb{.4}{1.15};
		\draw(.5,.5) node{#1};
		\draw(0,-.25) node{};
	\end{tikzpicture}
}

\newcommand{\planeTwo}[1]{
	\begin{tikzpicture}\scriptsize
		\draw[gray!50!white](0,0) -- (0,1) -- (1,1) -- (1,0) -- cycle;
		\nodegb{-.25}{.5};
		\nodegc{.4}{1.15};
		\draw(.5,.5) node{#1};
		\draw(0,-.25) node{};
	\end{tikzpicture}
}

\newcommand{\planeThree}[1]{
	\begin{tikzpicture}\scriptsize
		\draw[gray!50!white](0,0) -- (0,1) -- (1,1) -- (1,0) -- cycle;
		\nodegc{-.25}{.5};
		\nodega{.4}{1.15};
		\draw(.5,.5) node{#1};
		\draw(0,-.25) node{};
	\end{tikzpicture}
}

\newcommand{\tTabOne}{\scriptsize
	\begin{tikzpicture}
		\draw(1,1) node{
			\begin{tabular}{|c||c|c|}
				  \hline
				$T$ & $c_{1}$ & $c_{2}$ \\
				  \hline\hline
				$\bar{c}_{1}$ & $\times$ & $\times$ \\
				  \hline
				$\bar{c}_{2}$ & $\times$ & $\times$ \\
				  \hline
			\end{tabular}};
	\end{tikzpicture}}

\newcommand{\tTabTwo}{\scriptsize
	\begin{tikzpicture}
		\draw(1,1) node{
			\begin{tabular}{|c||c|c|}
				  \hline
				$T$ & $c_{1}$ & $c_{2}$ \\
				  \hline\hline
				$\bar{c}_{1}$ & $\times$ & $\times$ \\
				  \hline
				$\bar{c}_{2}$ & $\times$ & \\
				  \hline
			\end{tabular}};
	\end{tikzpicture}}

\newcommand{\tTabThree}{\scriptsize
	\begin{tikzpicture}
		\draw(1,1) node{
			\begin{tabular}{|c||c|c|}
				  \hline
				$T$ & $c_{1}$ & $c_{2}$ \\
				  \hline\hline
				$\bar{c}_{1}$ & $\times$ & $\times$ \\
				  \hline
				$\bar{c}_{2}$ & & \\
				  \hline
			\end{tabular}};
	\end{tikzpicture}}

\newcommand{\tTabFour}{\scriptsize
	\begin{tikzpicture}
		\draw(1,1) node{
			\begin{tabular}{|c||c|c|}
				  \hline
				$T$ & $c_{1}$ & $c_{2}$ \\
				  \hline\hline
				$\bar{c}_{1}$ & $\times$ & \\
				  \hline
				$\bar{c}_{2}$ & $\times$ & \\
				  \hline
			\end{tabular}};
	\end{tikzpicture}}

\newcommand{\tTabFive}{\scriptsize
	\begin{tikzpicture}
		\draw(1,1) node{
			\begin{tabular}{|c||c|c|}
				  \hline
				$T$ & $c_{1}$ & $c_{2}$ \\
				  \hline\hline
				$\bar{c}_{1}$ & $\times$ & \\
				  \hline
				$\bar{c}_{2}$ & & \\
				  \hline
			\end{tabular}};
	\end{tikzpicture}}

\newcommand{\tTabSix}{\scriptsize
	\begin{tikzpicture}
		\draw(1,1) node{
			\begin{tabular}{|c||c|c|}
				  \hline
				$T$ & $c_{1}$ & $c_{2}$ \\
				  \hline\hline
				$\bar{c}_{1}$ & & \\
				  \hline
				$\bar{c}_{2}$ & & \\
				  \hline
			\end{tabular}};
	\end{tikzpicture}}

\newcommand{\GalPsi}[6]{\scriptsize
	\begin{tikzpicture}[>=stealth']
		\draw(0,0) node[fill=black,circle,scale=.7](b1){};
		\draw(0,.5) node[fill=black,circle,scale=.7](b2){};
		\draw(0,1) node[fill=black,circle,scale=.7](b3){};
		\draw(1,0) node[draw,circle,scale=.7](a1){};
		\draw(1,.5) node[draw,circle,scale=.7](a2){};
		\draw(1,1) node[draw,circle,scale=.7](a3){};
		\draw(a1) -- (a2) -- (a3);
		\draw(b1) -- (b2) -- (b3);
		\draw[dashed,->](#1)--(#2);
		\draw[dashed,->](#3)--(#4);
		\draw[dashed,->](#5)--(#6);
	\end{tikzpicture}
}

\newcommand{\GalPhi}[6]{\scriptsize
	\begin{tikzpicture}[>=stealth']
		\draw(0,0) node[draw,circle,scale=.7](a1){};
		\draw(0,.5) node[draw,circle,scale=.7](a2){};
		\draw(0,1) node[draw,circle,scale=.7](a3){};
		\draw(1,0) node[fill=black,circle,scale=.7](b1){};
		\draw(1,.5) node[fill=black,circle,scale=.7](b2){};
		\draw(1,1) node[fill=black,circle,scale=.7](b3){};
		\draw(a1) -- (a2) -- (a3);
		\draw(b1) -- (b2) -- (b3);
		\draw[dashed,->](#1)--(#2);
		\draw[dashed,->](#3)--(#4);
		\draw[dashed,->](#5)--(#6);
	\end{tikzpicture}
}

\newcommand{\nodeee}[2]{
  \filldraw[green!67!black](#1,#2) arc (180:360:.1cm) -- cycle;
  \draw[green!67!black](#1,#2) arc(180:0:.1cm);}
\newcommand{\nodeef}[2]{
  \draw[green!67!black](#1,#2) arc (180:360:.1cm);
  \filldraw[green!67!black](#1,#2) arc(180:0:.1cm) -- cycle;}
\newcommand{\nodefe}[2]{
  \filldraw[black](#1,#2) arc (180:360:.1cm) -- cycle;
  \filldraw[black](#1,#2) arc(180:0:.1cm);}
\newcommand{\nodeff}[2]{
  \filldraw[black](#1,#2) arc (180:360:.1cm);
  \filldraw[black](#1,#2) arc(180:0:.1cm) -- cycle;}

\newcommand{\nOne}[2]{
	\begin{tikzpicture}[rotate=#1,scale=#2]
		\nodeee{0}{0};
		\nodeef{.5}{0};
		\nodefe{.25}{.5};
		\nodeff{.25}{1};
		\draw(.12,.1) -- (.3,.43);
		\draw(.57,.1) -- (.39,.43);
		\draw(.345,.55) -- (.345,.95);
	\end{tikzpicture}
}
\newcommand{\nTwo}[2]{
	\begin{tikzpicture}[rotate=#1,scale=#2]
		\nodeee{0}{.5};
		\nodeef{.5}{0};
		\nodefe{.5}{.5};
		\nodeff{.25}{1};
		\draw(.12,.6) -- (.3,.93);
		\draw(.595,.1) -- (.595,.43);
		\draw(.57,.6) -- (.39,.93);
	\end{tikzpicture}
}
\newcommand{\nThree}[2]{
	\begin{tikzpicture}[rotate=#1,scale=#2]
		\nodeee{0}{0};
		\nodeef{.5}{.5};
		\nodefe{0}{.5};
		\nodeff{.25}{1};
		\draw(.57,.6) -- (.39,.93);
		\draw(.095,.1) -- (.095,.43);
		\draw(.12,.6) -- (.3,.93);
	\end{tikzpicture}
}
\newcommand{\nFour}[2]{
	\begin{tikzpicture}[rotate=#1,scale=#2]
		\nodeee{0}{.5};
		\nodeef{.5}{0};
		\nodefe{.5}{.5};
		\nodeff{.25}{1};
		\draw(.595,.1) -- (.595,.43);
		\draw(.57,.6) -- (.39,.93);
	\end{tikzpicture}
}
\newcommand{\nFive}[2]{
	\begin{tikzpicture}[rotate=#1,scale=#2]
		\nodeee{0}{.5};
		\nodeef{.5}{.5};
		\nodefe{.25}{0};
		\nodeff{.25}{1};
		\draw(.12,.6) -- (.3,.93);
		\draw(.57,.6) -- (.39,.93);
		\draw(.345,.05) -- (.345,.95);
	\end{tikzpicture}
}
\newcommand{\nSix}[2]{
	\begin{tikzpicture}[rotate=#1,scale=#2]
		\nodeee{0}{0};
		\nodeef{.5}{.5};
		\nodefe{0}{.5};
		\nodeff{.25}{1};
		\draw(.095,.1) -- (.095,.43);
		\draw(.12,.6) -- (.3,.93);
	\end{tikzpicture}
}
\newcommand{\nSeven}[2]{
	\begin{tikzpicture}[rotate=#1,scale=#2]
		\nodeee{0}{.5};
		\nodeef{.5}{.5};
		\nodefe{.25}{0};
		\nodeff{.25}{1};
		\draw(.12,.6) -- (.3,.93);
		\draw(.57,.6) -- (.39,.93);
		\draw(.3,.07) -- (.12,.4);
	\end{tikzpicture}
}
\newcommand{\nEight}[2]{
	\begin{tikzpicture}[rotate=#1,scale=#2]
		\nodeee{0}{.5};
		\nodeef{.5}{.5};
		\nodefe{.25}{0};
		\nodeff{.25}{1};
		\draw(.57,.6) -- (.39,.93);
		\draw(.345,.05) -- (.345,.95);
	\end{tikzpicture}

}
\newcommand{\nNine}[2]{
	\begin{tikzpicture}[rotate=#1,scale=#2]
		\nodeee{0}{.5};
		\nodeef{.5}{.5};
		\nodefe{.25}{0};
		\nodeff{.25}{1};
		\draw(.12,.6) -- (.3,.93);
		\draw(.345,.05) -- (.345,.95);
	\end{tikzpicture}
}
\newcommand{\nTen}[2]{
	\begin{tikzpicture}[rotate=#1,scale=#2]
		\nodeee{0}{.5};
		\nodeef{.5}{.5};
		\nodefe{.25}{0};
		\nodeff{.25}{1};
		\draw(.12,.6) -- (.3,.93);
		\draw(.57,.6) -- (.39,.93);
		\draw(.39,.07) -- (.57,.4);
	\end{tikzpicture}
}
\newcommand{\nEleven}[2]{
	\begin{tikzpicture}[rotate=#1,scale=#2]
		\nodeee{0}{.5};
		\nodeef{.5}{.5};
		\nodefe{.25}{0};
		\nodeff{.25}{1};
		\draw(.57,.6) -- (.39,.93);
		\draw(.3,.07) -- (.12,.4);
		\draw(.345,.05) -- (.345,.95);
	\end{tikzpicture}
}
\newcommand{\nTwelve}[2]{
	\begin{tikzpicture}[rotate=#1,scale=#2]
		\nodeee{0}{.5};
		\nodeef{.5}{.5};
		\nodefe{.25}{0};
		\nodeff{.25}{1};
		\draw(.57,.6) -- (.39,.93);
		\draw(.39,.07) -- (.57,.4);
	\end{tikzpicture}
}
\newcommand{\nThirteen}[2]{
	\begin{tikzpicture}[rotate=#1,scale=#2]
		\nodeee{0}{.5};
		\nodeef{.5}{.5};
		\nodefe{.25}{0};
		\nodeff{.25}{1};
		\draw(.12,.6) -- (.3,.93);
		\draw(.57,.6) -- (.39,.93);
		\draw(.3,.07) -- (.12,.4);
		\draw(.39,.07) -- (.57,.4);
	\end{tikzpicture}
}
\newcommand{\nFourteen}[2]{
	\begin{tikzpicture}[rotate=#1,scale=#2]
		\nodeee{0}{.5};
		\nodeef{.5}{.5};
		\nodefe{.25}{0};
		\nodeff{.25}{1};
		\draw(.345,.05) -- (.345,.95);
	\end{tikzpicture}
}
\newcommand{\nFifteen}[2]{
	\begin{tikzpicture}[rotate=#1,scale=#2]
		\nodeee{0}{.5};
		\nodeef{.5}{.5};
		\nodefe{.25}{0};
		\nodeff{.25}{1};
		\draw(.12,.6) -- (.3,.93);
		\draw(.39,.07) -- (.57,.4);
		\draw(.345,.05) -- (.345,.95);
	\end{tikzpicture}
}

\newcommand{\bip}[4]{\scriptsize
	\begin{tikzpicture}[>=stealth']
		\draw(0,0) node[draw,circle](v1){#1};
		\draw(0,1) node[draw,circle](v2){#2};
		\draw(1,0) node[draw,circle](p1){#3};
		\draw(1,1) node[draw,circle](p2){#4};
		\draw[->] (v1) -- (p1);
		\draw[->] (v1) -- (p2);
		\draw[->] (v2) -- (p1);
		\draw[->] (v2) -- (p2);
	\end{tikzpicture}}

\newcommand{\rtabOne}{\scriptsize
	\begin{tikzpicture}
		\draw(1,1) node{
			\begin{tabular}{|c||c|c|}
				  \hline
				$R$ & $c_{1}$ & $c_{2}$ \\
				  \hline\hline
				$a_{1}$ & & \\
				  \hline
				$a_{2}$ & & \\
				  \hline
			\end{tabular}};
	\end{tikzpicture}}

\newcommand{\rtabTwo}{\scriptsize
	\begin{tikzpicture}
		\draw(1,1) node{
			\begin{tabular}{|c||c|c|}
				  \hline
				$R$ & $c_{1}$ & $c_{2}$ \\
				  \hline\hline
				$a_{1}$ & & $\times$ \\
				  \hline
				$a_{2}$ & & \\
				  \hline
			\end{tabular}};
	\end{tikzpicture}}

\newcommand{\rtabThree}{\scriptsize
	\begin{tikzpicture}
		\draw(1,1) node{
			\begin{tabular}{|c||c|c|}
				  \hline
				$R$ & $c_{1}$ & $c_{2}$ \\
				  \hline\hline
				$a_{1}$ & & \\
				  \hline
				$a_{2}$ & & $\times$ \\
				  \hline
			\end{tabular}};
	\end{tikzpicture}}

\newcommand{\rtabFour}{\scriptsize
	\begin{tikzpicture}
		\draw(1,1) node{
			\begin{tabular}{|c||c|c|}
				  \hline
				$R$ & $c_{1}$ & $c_{2}$ \\
				  \hline\hline
				$a_{1}$ & & $\times$ \\
				  \hline
				$a_{2}$ & & $\times$ \\
				  \hline
			\end{tabular}};
	\end{tikzpicture}}

\newcommand{\rtabFive}{\scriptsize
	\begin{tikzpicture}
		\draw(1,1) node{
			\begin{tabular}{|c||c|c|}
				  \hline
				$R$ & $c_{1}$ & $c_{2}$ \\
				  \hline\hline
				$a_{1}$ & $\times$ & $\times$ \\
				  \hline
				$a_{2}$ & & \\
				  \hline
			\end{tabular}};
	\end{tikzpicture}}

\newcommand{\rtabSix}{\scriptsize
	\begin{tikzpicture}
		\draw(1,1) node{
			\begin{tabular}{|c||c|c|}
				  \hline
				$R$ & $c_{1}$ & $c_{2}$ \\
				  \hline\hline
				$a_{1}$ & $\times$ & $\times$ \\
				  \hline
				$a_{2}$ & & $\times$ \\
				  \hline
			\end{tabular}};
	\end{tikzpicture}}

\newcommand{\rtabSeven}{\scriptsize
	\begin{tikzpicture}
		\draw(1,1) node{
			\begin{tabular}{|c||c|c|}
				  \hline
				$R$ & $c_{1}$ & $c_{2}$ \\
				  \hline\hline
				$a_{1}$ & & \\
				  \hline
				$a_{2}$ & $\times$ & $\times$ \\
				  \hline
			\end{tabular}};
	\end{tikzpicture}}

\newcommand{\rtabEight}{\scriptsize
	\begin{tikzpicture}
		\draw(1,1) node{
			\begin{tabular}{|c||c|c|}
				  \hline
				$R$ & $c_{1}$ & $c_{2}$ \\
				  \hline\hline
				$a_{1}$ & & $\times$ \\
				  \hline
				$a_{2}$ & $\times$ & $\times$ \\
				  \hline
			\end{tabular}};
	\end{tikzpicture}}

\newcommand{\rtabNine}{\scriptsize
	\begin{tikzpicture}
		\draw(1,1) node{
			\begin{tabular}{|c||c|c|}
				  \hline
				$R$ & $c_{1}$ & $c_{2}$ \\
				  \hline\hline
				$a_{1}$ & $\times$ & $\times$ \\
				  \hline
				$a_{2}$ & $\times$ & $\times$ \\
				  \hline
			\end{tabular}};
	\end{tikzpicture}}

\newcommand{\stabOne}{\scriptsize
	\begin{tikzpicture}
		\draw(1,1) node{
			\begin{tabular}{|c||c|c|}
				  \hline
				$S$ & $a_{1}$ & $a_{2}$ \\
				  \hline\hline
				$c_{1}$ & $\times$ & $\times$ \\
				  \hline
				$c_{2}$ & $\times$ & $\times$ \\
				  \hline
			\end{tabular}};
	\end{tikzpicture}}

\newcommand{\stabTwo}{\scriptsize
	\begin{tikzpicture}
		\draw(1,1) node{
			\begin{tabular}{|c||c|c|}
				  \hline
				$S$ & $a_{1}$ & $a_{2}$ \\
				  \hline\hline
				$c_{1}$ & $\times$ & $\times$ \\
				  \hline
				$c_{2}$ & & $\times$ \\
				  \hline
			\end{tabular}};
	\end{tikzpicture}}

\newcommand{\stabThree}{\scriptsize
	\begin{tikzpicture}
		\draw(1,1) node{
			\begin{tabular}{|c||c|c|}
				  \hline
				$S$ & $a_{1}$ & $a_{2}$ \\
				  \hline\hline
				$c_{1}$ & & $\times$ \\
				  \hline
				$c_{2}$ & & $\times$ \\
				  \hline
			\end{tabular}};
	\end{tikzpicture}}

\newcommand{\stabFour}{\scriptsize
	\begin{tikzpicture}
		\draw(1,1) node{
			\begin{tabular}{|c||c|c|}
				  \hline
				$S$ & $a_{1}$ & $a_{2}$ \\
				  \hline\hline
				$c_{1}$ & $\times$ & $\times$ \\
				  \hline
				$c_{2}$ & $\times$ & \\
				  \hline
			\end{tabular}};
	\end{tikzpicture}}

\newcommand{\stabFive}{\scriptsize
	\begin{tikzpicture}
		\draw(1,1) node{
			\begin{tabular}{|c||c|c|}
				  \hline
				$S$ & $a_{1}$ & $a_{2}$ \\
				  \hline\hline
				$c_{1}$ & $\times$ & $\times$ \\
				  \hline
				$c_{2}$ & & \\
				  \hline
			\end{tabular}};
	\end{tikzpicture}}

\newcommand{\stabSix}{\scriptsize
	\begin{tikzpicture}
		\draw(1,1) node{
			\begin{tabular}{|c||c|c|}
				  \hline
				$S$ & $a_{1}$ & $a_{2}$ \\
				  \hline\hline
				$c_{1}$ & & $\times$ \\
				  \hline
				$c_{2}$ & & \\
				  \hline
			\end{tabular}};
	\end{tikzpicture}}

\newcommand{\stabSeven}{\scriptsize
	\begin{tikzpicture}
		\draw(1,1) node{
			\begin{tabular}{|c||c|c|}
				  \hline
				$S$ & $a_{1}$ & $a_{2}$ \\
				  \hline\hline
				$c_{1}$ & $\times$ & \\
				  \hline
				$c_{2}$ & $\times$ & \\
				  \hline
			\end{tabular}};
	\end{tikzpicture}}

\newcommand{\stabEight}{\scriptsize
	\begin{tikzpicture}
		\draw(1,1) node{
			\begin{tabular}{|c||c|c|}
				  \hline
				$S$ & $a_{1}$ & $a_{2}$ \\
				  \hline\hline
				$c_{1}$ & $\times$ & \\
				  \hline
				$c_{2}$ & & \\
				  \hline
			\end{tabular}};
	\end{tikzpicture}}

\newcommand{\stabNine}{\scriptsize
	\begin{tikzpicture}
		\draw(1,1) node{
			\begin{tabular}{|c||c|c|}
				  \hline
				$S$ & $a_{1}$ & $a_{2}$ \\
				  \hline\hline
				$c_{1}$ & & \\
				  \hline
				$c_{2}$ & & \\
				  \hline
			\end{tabular}};
	\end{tikzpicture}}

\newcommand{\galPsi}[6]{\scriptsize
	\begin{tikzpicture}[>=stealth']
		\draw(0,0) node[fill=black,circle,scale=.7](b1){};
		\draw(0,.5) node[fill=black,circle,scale=.7](b2){};
		\draw(0,1) node[fill=black,circle,scale=.7](b3){};
		\draw(2,0) node[draw,circle,scale=.7](a1){};
		\draw(1.5,.5) node[draw,circle,scale=.7](a2){};
		\draw(2.5,.5) node[draw,circle,scale=.7](a3){};
		\draw(2,1) node[draw,circle,scale=.7](a4){};
		\draw(a1) -- (a2) -- (a4);
		\draw(a1) -- (a3) -- (a4);
		\draw(b1) -- (b2) -- (b3);
		\draw[dashed,->](#1)--(#2);
		\draw[dashed,->](#3)--(#4);
		\draw[dashed,->](#5)--(#6);
	\end{tikzpicture}
}

\newcommand{\galPsiTwo}[6]{\scriptsize
	\begin{tikzpicture}[>=stealth']
		\draw(0,0) node[fill=black,circle,scale=.7](b1){};
		\draw(0,.5) node[fill=black,circle,scale=.7](b2){};
		\draw(0,1) node[fill=black,circle,scale=.7](b3){};
		\draw(2,0) node[draw,circle,scale=.7](a1){};
		\draw(1.5,.5) node[draw,circle,scale=.7](a2){};
		\draw(2.5,.5) node[draw,circle,scale=.7](a3){};
		\draw(2,1) node[draw,circle,scale=.7](a4){};
		\draw(a1) -- (a2) -- (a4);
		\draw(a1) -- (a3) -- (a4);
		\draw(b1) -- (b2) -- (b3);
		\draw[dashed,->](#1)--(#2);
		\draw[dashed,->](#3) .. controls (1,.25) and (1.5,.25) .. (#4);
		\draw[dashed,->](#5)--(#6);
	\end{tikzpicture}
}

\newcommand{\galPhi}[8]{\scriptsize
	\begin{tikzpicture}[>=stealth']
		\draw(0,0) node[draw,circle,scale=.7](a1){};
		\draw(-.5,.5) node[draw,circle,scale=.7](a2){};
		\draw(.5,.5) node[draw,circle,scale=.7](a3){};
		\draw(0,1) node[draw,circle,scale=.7](a4){};
		\draw(2,0) node[fill=black,circle,scale=.7](b1){};
		\draw(2,.5) node[fill=black,circle,scale=.7](b2){};
		\draw(2,1) node[fill=black,circle,scale=.7](b3){};
		\draw(a1) -- (a2) -- (a4);
		\draw(a1) -- (a3) -- (a4);
		\draw(b1) -- (b2) -- (b3);
		\draw[dashed,->](#1)--(#2);
		\draw[dashed,->](#3)--(#4);
		\draw[dashed,->](#5)--(#6);
		\draw[dashed,->](#7)--(#8);
	\end{tikzpicture}
}

\newcommand{\galPhiTwo}[8]{\scriptsize
	\begin{tikzpicture}[>=stealth']
		\draw(0,0) node[draw,circle,scale=.7](a1){};
		\draw(-.5,.5) node[draw,circle,scale=.7](a2){};
		\draw(.5,.5) node[draw,circle,scale=.7](a3){};
		\draw(0,1) node[draw,circle,scale=.7](a4){};
		\draw(2,0) node[fill=black,circle,scale=.7](b1){};
		\draw(2,.5) node[fill=black,circle,scale=.7](b2){};
		\draw(2,1) node[fill=black,circle,scale=.7](b3){};
		\draw(a1) -- (a2) -- (a4);
		\draw(a1) -- (a3) -- (a4);
		\draw(b1) -- (b2) -- (b3);
		\draw[dashed,->](#1)--(#2);
		\draw[dashed,->](#3) .. controls (.5,.25) and (1,.25) .. (#4);
		\draw[dashed,->](#5)--(#6);
		\draw[dashed,->](#7)--(#8);
	\end{tikzpicture}
}
\title{Counting Proper Mergings of Chains and Antichains}

\begin{document}

\begin{abstract}
	A proper merging of two disjoint quasi-ordered sets $P$ and $Q$ is a quasi-order on the 
	union of $P$ and $Q$ such that the restriction to $P$ and $Q$ yields the original quasi-order 
	again and such that no elements of $P$ and $Q$ are identified. In this article, we consider the 
	cases where $P$ and $Q$ are chains, where $P$ and $Q$ are antichains, and where $P$ is an 
	antichain and $Q$ is a chain. We give formulas that determine the number of proper mergings in 
	all three cases, and introduce two new bijections from proper mergings of two chains to plane 
	partitions and from proper mergings of an antichain and a chain to monotone colorings of 
	complete bipartite digraphs. Additionally, we use these bijections to count the Galois connections
	between two chains, and between a chain and a Boolean lattice respectively.
\end{abstract}

\maketitle

\section{Introduction}
  \label{sec:introduction}
Given two quasi-ordered sets $(P,\parr)$ and $(Q,\qarr)$, a merging of $P$ and $Q$ is a quasi-order 
$\leftarrow$ on the union of $P$ and $Q$ such that the restriction of $\leftarrow$ to $P$ 
or $Q$ yields $\parr$ respectively $\qarr$ again. In other words, a merging of $P$ 
and $Q$ is a quasi-order on the union of $P$ and $Q$, which does not change the quasi-orders on $P$ and 
$Q$. 

In \cite{ganter11merging} a characterization of the set of mergings of two arbitrary quasi-ordered sets 
$P$ and $Q$ is given. In particular, it turns out that every merging $\leftarrow$ of $P$ and $Q$ can be 
uniquely described by two binary relations $R\subseteq P\times Q$ and $S\subseteq Q\times P$. The 
relation $R$ can be interpreted as a description, which part of $P$ is weakly below $Q$, and analogously 
the relation $S$ can be interpreted as a description, which part of $Q$ is weakly below $P$. A merging 
is called proper if $R\cap S^{-1}=\emptyset$, and hence if no element of $P$ is identified with an 
element of $Q$. 

The characterization in \cite{ganter11merging} uses techniques of Formal Concept Analysis 
(FCA, see \cite{ganter99formal}), a branch of mathematics, which investigates binary 
relations, so-called formal contexts, between two sets. The starting point of FCA is the construction 
of a closure system from such a formal context. Then, this closure system induces a complete lattice, 
when ordering the closures by inclusion. (A complete lattice is a possibly infinite lattice which has a
unique top and a unique bottom element.) The basic theorem of FCA states that every complete lattice 
can be derived from a formal context. In \cite{ganter11merging}, it was shown that the mergings
of two quasi-ordered sets $P$ and $Q$ form a distributive lattice, and can thus be described by a 
formal context. Notably, this formal context can be constructed easily from the quasi-orders
$\parr$ and $\qarr$. The proper mergings of $P$ and $Q$ form a distributive 
sublattice of the previous lattice.

Unfortunately, the formal context provides only very little information about the cardinality of its 
associated lattice. Hence, although the set of mergings of two quasi-orded sets $P$ and $Q$ can be 
described completely, not much is known about its cardinality. This article provides a first
enumerative analysis of the set of proper mergings of two special classes of quasi-ordered sets, namely 
chains and antichains. The actual genesis of this article was the observation that the number of proper
mergings of two $n$-chains is given by
\begin{align*}
	F_{\cc}(n)=\frac{(2n)!(2n+1)!}{(n!(n+1)!)^{2}}.
\end{align*}
It is stated in \cite{chen09pairs} that $F_{\cc}(n)$ also determines the number of plane partitions 
with $n$ rows, $n$ columns and largest part at most $2$. (See \cite{sloane}*{Sequence A000891} for 
some other objects counted by this number.) It is not hard to define a bijection between these plane
partitions, and the proper mergings of two $n$-chains, as will be described in 
\prettyref{sec:bijection_chains}. It is then straight-forward to extend this bijection to the set of plane 
partitions with $m$ rows, $n$ columns and largest part at most $2$, and the set 
of proper mergings of an $m$-chain and an $n$-chain. Since the number of such plane partitions can be 
derived from MacMahon's formula, see \eqref{eq:macmahon}, this bijection easily allows for counting the 
proper mergings of two chains. Interestingly, we can use this bijection for counting the Galois connections
between two chains. The key theorem for this correspondence is \cite{ganter99formal}*{Theorem~53}, which
states that the Galois connections between two concept lattices correspond to dual bonds between the 
corresponding formal contexts. 

After succeeding in enumerating proper mergings of chains, we became curious whether we can count 
proper mergings of two antichains in a similar way. Unfortunately, we cannot give a bijection between
the set of proper mergings of two antichains and any other known mathematical object. However, we 
are able to enumerate the proper mergings of two antichains with the help of a generating 
function, which was found by Christian Krattenthaler. See \prettyref{sec:antichains} for the details.

The third part of this article is devoted to the enumeration of proper mergings of an $m$-antichain 
and an $n$-chain. When computing the number of these proper mergings with the help of Daniel Borchmann's
FCA-tool \textsc{conexp-clj} \cite{conexp}, we recovered the sequence \cite{sloane}*{A085465}. The formula
generating this sequence is a special case of the following formula.
\begin{align*}
	F_{\ac,\cc}(m,n)=\sum_{i=1}^{n+1}{\Bigl(n+2-i)^{m}-(n+1-i)^{m}\Bigr)\cdot i^{m}}.
\end{align*}
It is stated in \cite{jovovic04antichains} that $F_{\ac,\cc}(m,n)$ also determines the number of
monotone $(n+1)$-colorings of the complete bipartite digraph $\vec{K}_{m,m}$. In
\prettyref{sec:bijection_antichain_chain}, we construct a bijection between the set of proper 
mergings of an $m$-antichain and an $n$-chain, and the set of monotone $(n+1)$-colorings of
$\vec{K}_{m,m}$. We can also use this bijection, in order to count the number of Galois connections between
a chain and a Boolean lattice. 

The precise statements of the results described in the previous paragraphs are the following. 

\begin{theorem}
  \label{thm:main_theorem}
	Let $\mathfrak{M}_{P,Q}^{\bullet}$ denote the set of proper mergings of two quasi-ordered sets
	$P$ and $Q$. 
	\begin{enumerate}
		\item[(i)] Let $P$ and $Q$ be chains. If $\lvert P\rvert=m,\lvert Q\rvert=n$, then
		  \begin{align*}
			\lvert\mathfrak{M}_{P,Q}^{\bullet}\rvert = 
			  \frac{1}{n+m+1}\binom{n+m+1}{m+1}\binom{n+m+1}{m}.
		  \end{align*}
		\item[(ii)] Let $P$ and $Q$ be antichains. If $\lvert P\rvert=m,\lvert Q\rvert=n$, 
		  then
		  \begin{align*}
			\lvert\mathfrak{M}_{P,Q}^{\bullet}\rvert = 
			  \sum_{n_{1}+m_{1}+k_{1}=m}{\binom{m}{n_{1},m_{1},k_{1}}(-1)^{k_{1}}
			  \Bigl(2^{n_{1}}+2^{m_{1}}-1\Bigr)^{n}}.
		  \end{align*}
		\item[(iii)] Let $P$ be an antichain, and let $Q$ be a chain. 
		  If $\lvert P\rvert=m,\lvert Q\rvert=n$, then
		  \begin{align*}
			\lvert\mathfrak{M}_{P,Q}^{\bullet}\rvert =
			  \sum_{i=1}^{n+1}{\Bigl((n+2-i)^{m}-(n+1-i)^{m}\Bigr)i^{m}}.
		  \end{align*}
	\end{enumerate}
\end{theorem}
In \prettyref{thm:main_theorem}~(iii), we need to be careful with the case $m=0$. In this case, there 
appears a term of the form ``$0^{0}$`` in the sum. Since there is exactly one proper merging of an empty
antichain and some chain, we need to interpret this term as being equal to zero. 

This article is organized as follows: in \prettyref{sec:preliminaries}, we give a short introduction 
to Formal Concept Analysis in order to make the reader familiar with notions such as cross-table,
intent, extent, bond, and other terminology from FCA. Moreover, we formally define mergings of two
quasi-ordered sets. In \prettyref{sec:chains}, we define the bijection between proper mergings of two
chains, and plane partitions with largest part at most $2$. We conclude \prettyref{thm:main_theorem}~(i) 
in \prettyref{sec:counting_mergings_chains}, and exploit this bijection in order to count the Galois
connections between two chains in \prettyref{sec:galois_connections_chains}. In 
\prettyref{sec:antichains}, we compute the generating function for the proper mergings of two 
antichains and conclude \prettyref{thm:main_theorem}~(ii). In \prettyref{sec:antichain_chain}, we 
construct the bijection between proper mergings of an antichain and a chain, and monotone colorings of a 
complete bipartite digraph. We conclude \prettyref{thm:main_theorem}~(iii) in 
\prettyref{sec:bijection_antichain_chain}, and exploit this bijection in order to count the Galois 
connections between chains and Boolean lattices in \prettyref{sec:galois_connections_boolean_chain}. 

\section{Preliminaries}
  \label{sec:preliminaries}
In this section we recall the basic notations and definitions needed in this article. For a
detailed introduction to Formal Concept Analysis, we refer to \cite{ganter99formal}. 

\subsection{Formal Concept Analysis}
  \label{sec:formal_concept_analysis}
The theory of Formal Concept Ana\-lysis (FCA) was introduced in the 1980s by Rudolf Wille 
(see \cite{wille82restructuring}) as an approach to restructure lattice theory. The initial goal was 
to interpret lattices as hierarchies of concepts and thus to give meaning to the lattice elements in 
a fixed context. Such a \emph{formal context} is a triple $(G,M,I)$, where $G$ is a set of 
so-called \emph{objects}, $M$ is a set of so-called \emph{attributes} and $I\subseteq G\times M$ is
a binary relation that describes whether an object \emph{has} an attribute. Given a formal context
$\mathbb{K}=(G,M,I)$, we define two derivation operators
\begin{align}
	\label{eq:int} (\cdot)^{I}:\wp(G)\rightarrow\wp(M), & 
	  \quad A\mapsto A^{I}=\{m\in M\mid g\;I\;m\;\text{for all}\;g\in A\},\\
	\label{eq:ext} (\cdot)^{I}:\wp(M)\rightarrow\wp(G), & 
	  \quad B\mapsto B^{I}=\{g\in G\mid g\;I\;m\;\text{for all}\;m\in B\},
\end{align}
where $\wp$ denotes the power set. The notation $g\;I\;m$ is to be understood as $(g,m)\in I$. It
shall be mentioned that these derivation operators form a Galois connection between $\wp(G)$ and
$\wp(M)$, and hence, the composition $(\cdot)^{II}$ is a closure operator on $\wp(G)$ respectively
on $\wp(M)$. (See \prettyref{sec:galois_connections_chains} for an explicit definition of Galois 
connections.) We notice the natural duality between these operators, which justifies the use of 
the same symbol for both of them.

Let now $A\subseteq G$, and $B\subseteq M$. The pair $\mathfrak{b}=(A,B)$ is called 
\emph{formal concept of $\KK$} if $A^{I}=B$ and $B^{I}=A$. In this case, we call $A$ the 
\emph{extent} and $B$ the \emph{intent of $\mathfrak{b}$}. It can easily be seen that for every 
$A\subseteq G$, and $B\subseteq M$, the pairs $\bigl(A^{II},A^{I}\bigr)$ and 
$\bigl(B^{I},B^{II}\bigr)$ are formal concepts, respectively. Conversely, every formal concept of 
$\KK$ can be written in such a way. Thus, every formal concept of a given formal context can be seen 
from an extensional (``Which objects does the concept describe?'') as well as an intensional 
(``Which attributes describe the concept?'') viewpoint. We denote the set of all formal concepts of 
$\KK$ by $\mathfrak{B}(\KK)$, and define a partial order on $\mathfrak{B}(\KK)$ by
\begin{multline}
	(A_{1},B_{1})\leq(A_{2},B_{2})\quad\text{if and only if} \\ 
		A_{1}\subseteq A_{2}\quad(\text{or equivalently}\;B_{1}\supseteq B_{2}).
\end{multline}
Let $\CL(\KK)$ denote the poset $\bigl(\mathfrak{B}(\KK),\leq\bigr)$. The basic theorem of FCA (see 
\cite{ganter99formal}*{Theorem~3}) states that $\CL(\KK)$ is a lattice, the so-called 
\emph{concept lattice of $\KK$}. Moreover, every finite lattice is a concept lattice\footnote{More 
precisely, the basic theorem of FCA states that every \emph{complete lattice} is a concept lattice. 
A complete lattice is a (possibly infinite) lattice which has a unique minimal and a unique maximal
element. In particular, every finite lattice is a complete lattice.}. This result implies that every 
element of a finite lattice can be interpreted as a closure of a suitable closure system. 

Usually, a formal context is represented by a cross-table, where the rows represent the 
objects and the columns represent the attributes. The cell in row $g$ and column $m$ contains a 
cross if and only if $g\;I\;m$. See \prettyref{fig:formal_contexts} for two small examples. The 
reader is encouraged to compute the concept lattices of both formal contexts in order to see that 
these lattices are indeed isomorphic to a $4$-chain, respectively a Boolean lattice with eight 
elements. 

\begin{figure}
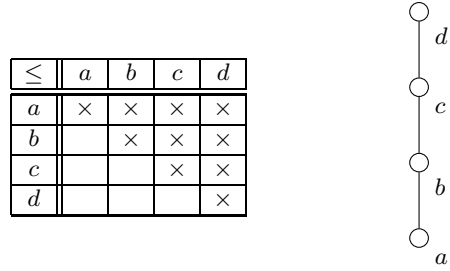
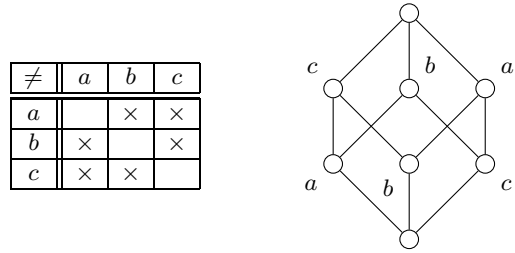

	\centering
	\subfigure[The cross-table representing the formal context associated to a $4$-chain.]{
		\label{fig:chain_context}
		\chainContext
	}
	\subfigure[The cross-table representing the formal context associated to a Boolean lattice
	  with eight elements.]{
		\label{fig:boolean_context}
		\booleanContext
	}
	\caption{Two examples for a formal context associated to a lattice.}
	\label{fig:formal_contexts}
\end{figure}

For every context $\KK=(G,M,I)$, there are two maps
\begin{align}
	\label{eq:object}
	  \gamma & :G\rightarrow\CL(\KK), && \hspace*{-2cm} g\mapsto \bigl(\{g\}^{II},\{g\}^{I}\bigr),
	  \quad\text{and}\\
	\label{eq:attribute}
	  \mu & :M\rightarrow\CL(\KK), && \hspace*{-2cm} m\mapsto \bigl(\{m\}^{I},\{m\}^{II}\bigr),
\end{align}
which map each object, respectively attribute, to its corresponding formal concept. It is common sense
in FCA to label the Hasse diagram of $\CL(\KK)$ in the following way: the node representing a formal 
concept $\mathfrak{b}\in\mathfrak{B}(\KK)$ is labeled with the object $g$ (or with the attribute $m$) 
if and only if $\mathfrak{b}=\gamma g$ (or $\mathfrak{b}=\mu m$). Object labels are attached 
below the nodes in the Hasse diagram, and attribute labels above. In this presentation, the extent 
(intent) of a formal concept corresponds to the labels weakly below (weakly above) this formal 
concept in the Hasse diagram of $\CL(\KK)$. (In \prettyref{fig:chain_context}, however, we omitted 
the attribute labels, since they would be attached to the same formal concept as the corresponding 
object label.)

There is yet another way to interpret formal contexts. Let $(P,\pleq)$ be a poset. Then, 
$(P,P,\pleq)$ is a formal context and its cross-table corresponds to the incidence matrix of
$(P,\pleq)$, which means that we can read the order-relation of $(P,\pleq)$ from the cross-table. 
Moreover, the concept lattice $\CL(P,P,\leq)$ is isomorphic to the smallest (complete) lattice that 
contains $(P,\pleq)$ as a subposet, the so-called \emph{Dedekind-MacNeille completion of $(P,\pleq)$}. 
However, not every crosstable of a formal context $(P,P,I)$ can be interpreted as the incidence 
matrix of a partial order on $P$. (For instance, the cross-table shown in
\prettyref{fig:boolean_context} does not correspond to a partial order on the set $\{a,b,c\}$.)

In the remainder of this article, we will usually represent posets (and binary relations in general) 
by the cross-table of the corresponding formal context. Whenever we speak of a row or column in combination 
with a poset element $p\in P$, we mean the corresponding set $\{p\}^{\pleq}$ in the sense
of \eqref{eq:int} (respectively \eqref{eq:ext}).

\subsection{Bonds and Mergings}
  \label{sec:bonds_mergings}
Let $\KK_{1}=(G_{1},M_{1},I_{1}),\KK_{2}=(G_{2},M_{2},I_{2})$ be formal contexts. A binary
relation $R\subseteq G_{1}\times M_{2}$ is called \emph{bond from $\KK_{1}$ to $\KK_{2}$} if for 
every object $g\in G_{1}$, the row $\{g\}^{R}$ is an intent in $\KK_{2}$ and for every 
$m\in M_{2}$, the column $\{m\}^{R}$ is an extent in $\KK_{1}$. 

Now let $(P,\parr)$ and $(Q,\qarr)$ be disjoint quasi-ordered sets. Let $R\subseteq P\times Q$, and 
$S\subseteq Q\times P$. Define a relation $\leftarrow_{R,S}$ on $P\cup Q$ as
\begin{multline}
	\label{eq:merging}
	p\leftarrow_{R,S} q\quad\text{if and only if}\\
	p\parr q\;\;\text{or}\;\;p\qarr q\;\;\text{or}\;\;p\;R\;q\;\;\text{or}\;\;p\;S\;q,
\end{multline}
for all $p,q\in P\cup Q$. The pair $(R,S)$ is called \emph{merging of $P$ and $Q$} if 
$(P\cup Q,\leftarrow_{R,S})$ is a quasi-ordered set. Moreover, a merging is called \emph{proper} if 
$R\cap S^{-1}=\emptyset$. Since for fixed quasi-ordered sets $(P,\parr)$ and $(Q,\qarr)$ the relation 
$\leftarrow_{R,S}$ is uniquely determined by $R$ and $S$, we refer to $\leftarrow_{R,S}$ as a 
(proper) merging of $P$ and $Q$ as well. Let $\circ$ denote the relational 
product\footnote{Let $R\subseteq A\times B$ and $S\subseteq B\times C$ be relations between sets 
$A,B,C$. The \emph{relational product} is the relation $R\circ S\subseteq A\times C$ that is given by 
$R\circ S=\{(a,c)\mid (a,b)\in R\;\text{and}\;(b,c)\in S\;\text{for some}\;b\in B\}$.}.

\begin{proposition}[\cite{ganter11merging}*{Proposition~2}]
  \label{prop:classification_mergings}
	Let $(P,\parr)$ and $(Q,\qarr)$ be disjoint quasi-ordered sets, and let 
	$R\subseteq P\times Q$, and $S\subseteq Q\times P$. The pair $(R,S)$ is a merging of $P$ and $Q$ 
	if and only if all of the following properties are satisfied:
	\begin{enumerate}
		\item $R$ is a bond from $(P,P,\not\rightarrow_{P})$ to $(Q,Q,\not\rightarrow_{Q})$,
		\item $S$ is a bond from $(Q,Q,\not\rightarrow_{Q})$ to $(P,P,\not\rightarrow_{P})$,
		\item $R\circ S$ is contained in $\parr$, and
		\item $S\circ R$ is contained in $\qarr$.
	\end{enumerate}
	Moreover, the relation $\leftarrow_{R,S}$ as defined in \eqref{eq:merging} is antisymmetric if 
	and only if $\parr$ and $\qarr$ are both antisymmetric and $R\cap S^{-1}=\emptyset$. 
\end{proposition}
In the case that $P$ and $Q$ are posets, this proposition implies that $(P\cup Q,\leftarrow_{R,S})$ is 
a poset again if and only if $(R,S)$ is a proper merging of $P$ and $Q$. 

Denote the set of mergings of $P$ and $Q$ by 
$\mathfrak{M}_{P,Q}$, and define a partial order on $\mathfrak{M}_{P,Q}$ by 
\begin{align}
	(R_{1},S_{1})\preceq(R_{2},S_{2})\quad\text{if and only if}\quad 
	  R_{1}\subseteq R_{2}\;\text{and}\;S_{1}\supseteq S_{2}.
\end{align}
It was shown in \cite{ganter11merging}*{Theorem~1} that 
$\bigl(\mathfrak{M}_{P,Q},\preceq\bigr)$ is a distributive lattice, where 
$(\emptyset,Q\times P)$ is the unique minimal element, and $(P\times Q,\emptyset)$ the unique 
maximal element. Let $\mathfrak{M}_{P,Q}^{\bullet}\subseteq\mathfrak{M}_{P,Q}$ denote the
set of all proper mergings of $P$ and $Q$. It is also stated in 
\cite{ganter11merging}*{Theorem~1} that $(\mathfrak{M}_{P,Q}^{\bullet},\preceq)$ is a 
(complete) sublattice of $(\mathfrak{M}_{P,Q},\preceq)$, which is still distributive. 

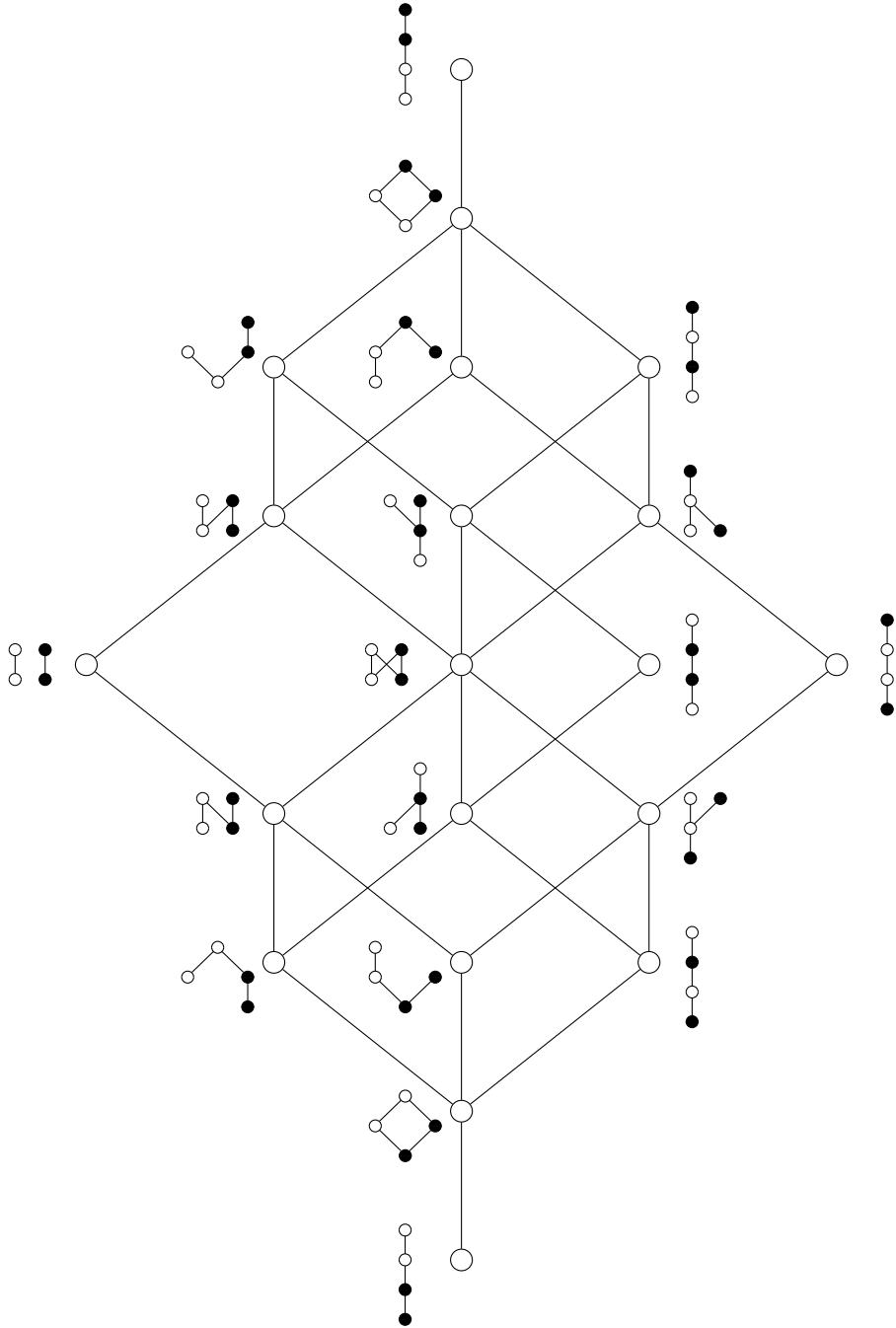
\begin{figure}
	\centering
	\begin{tikzpicture}\small
	\def\x{2.5};
	\def\y{2};
	\def\r{.8};
	\def\rs{.8};
	\draw(3*\x,1*\y) node[draw,circle,scale=\r](n1){};
	\draw(3*\x,2*\y) node[draw,circle,scale=\r](n2){};
	\draw(2*\x,3*\y) node[draw,circle,scale=\r](n3){};
	\draw(3*\x,3*\y) node[draw,circle,scale=\r](n4){};
	\draw(4*\x,3*\y) node[draw,circle,scale=\r](n5){};
	\draw(2*\x,4*\y) node[draw,circle,scale=\r](n6){};
	\draw(3*\x,4*\y) node[draw,circle,scale=\r](n7){};
	\draw(4*\x,4*\y) node[draw,circle,scale=\r](n8){};
	\draw(1*\x,5*\y) node[draw,circle,scale=\r](n9){};
	\draw(3*\x,5*\y) node[draw,circle,scale=\r](n10){};
	\draw(4*\x,5*\y) node[draw,circle,scale=\r](n11){};
	\draw(5*\x,5*\y) node[draw,circle,scale=\r](n12){};
	\draw(2*\x,6*\y) node[draw,circle,scale=\r](n13){};
	\draw(3*\x,6*\y) node[draw,circle,scale=\r](n14){};
	\draw(4*\x,6*\y) node[draw,circle,scale=\r](n15){};
	\draw(2*\x,7*\y) node[draw,circle,scale=\r](n16){};
	\draw(3*\x,7*\y) node[draw,circle,scale=\r](n17){};
	\draw(4*\x,7*\y) node[draw,circle,scale=\r](n18){};
	\draw(3*\x,8*\y) node[draw,circle,scale=\r](n19){};
	\draw(3*\x,9*\y) node[draw,circle,scale=\r](n20){};
	\draw(n1) -- (n2);
	\draw(n2) -- (n3);
	\draw(n2) -- (n4);
	\draw(n2) -- (n5);
	\draw(n3) -- (n6);
	\draw(n3) -- (n7);
	\draw(n4) -- (n6);
	\draw(n4) -- (n8);
	\draw(n5) -- (n7);
	\draw(n5) -- (n8);
	\draw(n6) -- (n9);
	\draw(n6) -- (n10);
	\draw(n7) -- (n10);
	\draw(n7) -- (n11);
	\draw(n8) -- (n10);
	\draw(n8) -- (n12);
	\draw(n9) -- (n13);
	\draw(n10) -- (n13);
	\draw(n10) -- (n14);
	\draw(n10) -- (n15);
	\draw(n11) -- (n14);
	\draw(n12) -- (n15);
	\draw(n13) -- (n16);
	\draw(n13) -- (n17);
	\draw(n14) -- (n16);
	\draw(n14) -- (n18);
	\draw(n15) -- (n17);
	\draw(n15) -- (n18);
	\draw(n16) -- (n19);
	\draw(n17) -- (n19);
	\draw(n18) -- (n19);
	\draw(n19) -- (n20);
	\draw(2.7*\x,.9*\y) node{\lOne{0}{\rs}};
	\draw(2.7*\x,1.9*\y) node{\lTwo{0}{\rs}};
	\draw(1.7*\x,2.9*\y) node{\lThree{0}{\rs}};
	\draw(2.7*\x,2.9*\y) node{\lFour{0}{\rs}};
	\draw(4.23*\x,2.9*\y) node{\lFive{0}{\rs}};
	\draw(1.7*\x,4*\y) node{\lSix{0}{\rs}};
	\draw(2.7*\x,4.1*\y) node{\lSeven{0}{\rs}};
	\draw(4.3*\x,3.9*\y) node{\lEight{0}{\rs}};
	\draw(.7*\x,5*\y) node{\lNine{0}{\rs}};
	\draw(2.6*\x,5*\y) node{\lTen{0}{\rs}};
	\draw(4.23*\x,5*\y) node{\lEleven{0}{\rs}};
	\draw(5.27*\x,5*\y) node{\lTwelve{0}{\rs}};
	\draw(1.7*\x,6*\y) node{\reflectbox{\lSix{180}{\rs}}};
	\draw(2.7*\x,5.9*\y) node{\reflectbox{\lSeven{180}{\rs}}};
	\draw(4.3*\x,6.1*\y) node{\reflectbox{\lEight{180}{\rs}}};
	\draw(1.7*\x,7.1*\y) node{\reflectbox{\lThree{180}{\rs}}};
	\draw(2.7*\x,7.1*\y) node{\reflectbox{\lFour{180}{\rs}}};
	\draw(4.23*\x,7.1*\y) node{\reflectbox{\lFive{180}{\rs}}};
	\draw(2.7*\x,8.15*\y) node{\reflectbox{\lTwo{180}{\rs}}};
	\draw(2.7*\x,9.1*\y) node{\reflectbox{\lOne{180}{\rs}}};
\end{tikzpicture}
	\caption{The lattice of proper mergings of two $2$-chains.}
	\label{fig:mergings_two_chains}
\end{figure}

\prettyref{fig:mergings_two_chains} shows the lattice of proper 
mergings of two $2$-chains, where the nodes are labeled by the corresponding proper mergings. 

\section{Proper Mergings of two Chains}
  \label{sec:chains}
In the first part of this article, we provide a closed formula for the number of proper mergings
of two chains. In particular, we give a bijective proof of the following theorem.

\begin{theorem}
  \label{thm:enumeration_chains}
	Let $m,n\in\mathbb{N}$ and let $\mathfrak{C}_{m,n}^{\bullet}$ denote the set of proper 
	mergings of an $n$-chain and an $m$-chain. Then,
	\begin{align*}
		\left\lvert\mathfrak{C}_{m,n}^{\bullet}\right\rvert=
		  \frac{1}{n+m+1}\binom{n+m+1}{m+1}\binom{n+m+1}{m}.
	\end{align*}
\end{theorem}

In addition, we exploit the bijection constructed in this section to count the number of Galois 
connections between two chains. 

We start with some definitions. Let $C=\{c_{1},c_{2},\ldots,c_{n}\}$ be a set. Consider the $n$-chain
$(C,\leq)$, where the order $\leq$ is indicated by the indices, namely $c_{i}\leq c_{j}$ if and only
if $i\leq j$. In the remainder of this section, we abbreviate the poset $(C,\leq)$ by $\cc$. The
corresponding formal context $(C,C,\leq)$ will be denoted by $\KK(\cc)$. The formal context 
$(C,C,\not\geq)$ -- the so-called \emph{contraordinal scale of $\cc$} -- will be denoted by
$\CC(\cc)$.

\subsection{Intents and Extents of $\CC(\cc)$}
  \label{sec:intents_extents}
If $\cc=(C,\leq)$ is an $n$-chain, we can convince ourselves that we can write the corresponding 
cross-table of $\KK(\cc)$ in a triangular shape, as indicated in \prettyref{fig:chain_context}. 
Since the elements in $\cc$ are pairwise comparable, we have for all $c,c'\in C$ that $c\not\geq c'$ 
if and only if $c<c'$. Hence, the cross-table of the context $\CC(\cc)$ is that of $\KK(\cc)$ 
without crosses on the main diagonal. Thus, for every $i\in\{2,3,\ldots,n\}$ the set 
$\{c_{i},c_{i+1},\ldots,c_{n}\}$ is a row (and thus an intent) of $\CC(\cc)$. At the same time, 
for every $i\in\{1,2,\ldots,n-1\}$, the set $\{c_{1},c_{2},\ldots,c_{i}\}$ is a column (and thus 
an extent) of $\CC(\cc)$. By definition, the empty set and $C$ itself are both intents and extents 
of $\CC(\cc)$. (This follows, since the empty set is an extent (intent) of a formal context if and 
only if there is no full row (column). The set of objects (attributes) is an extent (intent) of 
every formal context.) This means that for every $i\in\{0,1,\ldots,n\}$, the set 
$\{c_{1},c_{2},\ldots,c_{i}\}$ is an extent of $\CC(\cc)$, and hence $\CL\bigl(\CC(\cc)\bigr)$ is 
isomorphic to an $n+1$-chain. (The case $i=0$ is to be interpreted as the empty set.) See 
\prettyref{fig:contraordinal_context} for an illustration.

\begin{figure}
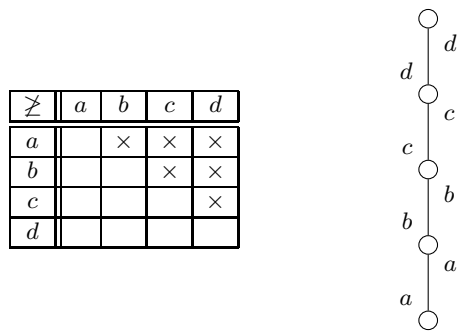

	\centering
	\contraordinalChain
	\caption{The formal context of the contraordinal scale of the $4$-chain from 
	  \prettyref{fig:chain_context}. The concept lattice of this formal context is the 
	  $5$-chain on the right.}
	\label{fig:contraordinal_context}
\end{figure}

\subsection{A Bijection between Plane Partitions and Proper Mergings of Two Chains}
  \label{sec:bijection_chains}
Let us recall that a \emph{plane partition} $\pi=(\pi_{i,j})_{i,j\geq 1}$ is an 
array of nonnegative integers that is weakly decreasing along rows and columns and has only finitely 
many nonzero entries. An entry $\pi_{i,j}$ is called \emph{part of $\pi$}. (We refer the reader to 
\cite{stanley01enumerative}*{Sections~7.20~and~7.21} for more information on plane partitions.)
The next definition is central for this section.

\begin{definition}
  \label{def:bijection_chains}
	Let $C_{1}=\{a_{1},a_{2},\ldots,a_{m}\}$, and $C_{2}=\{b_{1},b_{2},\ldots,b_{n}\}$ be sets. 
	Consider the chains $\cc_{1}=(C_{1},\leq_{1})$, and $\cc_{2}=(C_{2},\leq_{2})$, where the
	order relations are determined by the indices of the corresponding sets. Let $\pi$ be a plane 
	partition with $m$ rows, $n$ columns, and largest part at most $2$. Define relations 
	$R_{\pi}\subseteq C_{1}\times C_{2}$ and $S_{\pi}\subseteq C_{2}\times C_{1}$ by 
	\begin{align}
	  \label{eq:mergings_r}
		a_{i}\;R_{\pi}\;b_{n-j+1} & \quad\text{if and only if}\quad \pi_{i,j}=2,\;\text{and}\\
	  \label{eq:mergings_s}
		b_{n-j+1}\;S_{\pi}\;a_{i} & \quad\text{if and only if}\quad \pi_{i,j}=0,
	\end{align}
	where $1\leq i\leq m$ and $1\leq j\leq n$.
\end{definition}

\begin{figure}
	\centering
	\exampleOnePartition
	\caption{A plane partition with five rows, six columns and largest part $2$.}
	\label{fig:plane_partition}
\end{figure}

\begin{figure}
	\centering
	\begin{tikzpicture}
		\draw(0,0) node{\exampleOneRTable};
		\draw(6,0) node{\exampleOneSTable};
	\end{tikzpicture}
	\caption{The relations $R$ and $S$ induced by the plane partition in 
	  \prettyref{fig:plane_partition}.}
	\label{fig:merging_from_partition_cxt}
\end{figure}

\prettyref{fig:plane_partition} shows a plane partition with five rows, six columns, and largest part 
$2$. \prettyref{fig:merging_from_partition_cxt} shows the corresponding relations $R$ and $S$ in the
sense of the previous definition.

\begin{lemma}
  \label{lem:create_merging_chains}
	The relations $R_{\pi}$ and $S_{\pi}$ from \prettyref{def:bijection_chains} form a 
	proper merging of $\cc_{1}$ and $\cc_{2}$.
\end{lemma}
\begin{proof}
	It is sufficient to prove that $R_{\pi}$ and $S_{\pi}$ satisfy the conditions (1)--(4) in
	\prettyref{prop:classification_mergings}. First, let $a_{i},a_{j}\in C_{1}$, with 
	$(a_{i},a_{j})\in R_{\pi}\circ S_{\pi}$. By definition, there must be some $b_{k}\in C_{2}$ 
	satisfying $\pi_{i,k}=2$ and $\pi_{j,k}=0$. Since $\pi$ is a plane partition (and hence weakly
	decreasing along the columns), we can conclude that $i<j$, and hence $a_{i}<a_{j}$, which
	proves condition (3). Now let $b_{i},b_{j}\in C_{2}$ with 
	$(b_{i},b_{j})\in S_{\pi}\circ R_{\pi}$. By definition, there must be some $a_{k}\in C_{1}$ 
	satisfying $\pi_{k,n-i+1}=0$ and $\pi_{k,n-j+1}=2$. Again we can conclude that $i<j$, and thus 
	$b_{i}<b_{j}$, which proves condition (4).

	Now we need to show that $R_{\pi}$ is a bond from $\CC(\cc_{1})$ to $\CC(\cc_{2})$ and 
	$S_{\pi}$ is a bond from $\CC(\cc_{2})$ to $\CC(\cc_{1})$. Hence, we need to show that every 
	row in $R_{\pi}$ is an intent of $\CC(\cc_{2})$, and every column in $R_{\pi}$ is an extent 
	of $\CC(\cc_{1})$. First we notice that for every $i\in\{1,2,\ldots,n\}$, the set 
	$\{a_{i}\}^{R_{\pi}}$ consists of all $b_{j}\in C_{2}$ such that $\pi_{i,j}=2$. Since, $\pi$ 
	is a plane partition, we can conclude that $\{a_{i}\}^{R_{\pi}}$ is of the form 
	$\{b_{k},b_{k+1},\ldots,b_{n}\}$ for some $k\in\{1,2,\ldots,n+1\}$. (The case $k=n+1$ is to be 
	interpreted as the empty set.) The reasoning in the beginning of this section shows that each 
	such set is indeed an intent of $\CC(\cc_{2})$. Similarly, we see that for every $b\in C_{2}$, 
	the set $\{b\}^{R_{\pi}}$ is of the form $\{a_{1},a_{2},\ldots,a_{k}\}$ for some 
	$k\in\{0,1,\ldots,n\}$. (The case $k=0$ is to be interpreted as the empty set.) By the same 
	argument as before, we see that these indeed are extents of $\CC(\cc_{2})$, which proves 
	condition (1). To show that $S_{\pi}$ is a bond from $\CC(\cc_{2})$ to $\CC(\cc_{1})$, we 
	notice that the rows in $S_{\pi}$ must correspond to intents of $\CC(\cc_{1})$ and the columns 
	of $S_{\pi}$ must correspond to extents of $\CC(\cc_{2})$. Thus, condition $(2)$ can be shown 
	analogously to the previous case. 
	
	Finally, since every cell is labeled by a unique value, we can conclude that 
	$R_{\pi}\cap S_{\pi}^{-1}=\emptyset$, which makes $(R_{\pi},S_{\pi})$ a proper merging of
	$\cc_{1}$ and $\cc_{2}$.
\end{proof}

\begin{figure}
	\centering
	\exampleOnePoset
	\caption{The proper merging of a $5$-chain and a $6$-chain defined by 
	  the relations given in \prettyref{fig:merging_from_partition_cxt}.}
	\label{fig:merging_from_partition}      
\end{figure}

\prettyref{fig:merging_from_partition} shows the poset corresponding to the proper merging shown in
\prettyref{fig:merging_from_partition_cxt}. We can conclude the following theorem.

\begin{theorem}
  \label{thm:bijection_chains}
	Let $\text{PP}_{m,n}^{(2)}$ denote the set of plane partitions with $m$ rows, $n$ columns and 
	largest part at most $2$. Let $\mathfrak{C}_{m,n}^{\bullet}$ denote the set of proper mergings 
	of an $m$-chain and an $n$-chain. Then, the correspondence described in 
	\prettyref{def:bijection_chains} is a bijection between $\text{PP}_{m,n}^{(2)}$ and 
	$\mathfrak{C}_{m,n}^{\bullet}$.
\end{theorem}
\begin{proof}
	\prettyref{lem:create_merging_chains} makes immediately clear that each such plane partition 
	induces a proper merging of an $m$-chain and an $n$-chain. 

	Conversely, let $(R,S)$ be a proper merging of an $m$-chain and an $n$-chain. Let 
	$\pi_{(R,S)}$ be the $(m\times n)$-array, whose parts $\pi_{i,j}$ are defined by
	\begin{align*}
		\pi_{i,j}=\begin{cases}
			0, & \text{if}\;b_{n-j+1}\;S\;a_{i},\\
			2, & \text{if}\;a_{i}\;R\;b_{n-j+1},\\
			1, & \text{otherwise}.
		\end{cases}
	\end{align*}
	Since $(R,S)$ is a proper merging, no cell is labeled twice. Condition~(1) in 
	\prettyref{prop:classification_mergings}  implies that if more than one $2$ appears in a row 
	or column of $\pi_{(R,S)}$, these $2$'s appear consecutively. Moreover, it follows that a row 
	(or column), which contains a $2$, contains a $2$ in its first cell. Condition~(2) in 
	\prettyref{prop:classification_mergings} implies the analogous properties for $0$'s, in 
	particular that a row (or column) that contains a $0$, contains a $0$ in its last cell. 
	Condition~(3) in \prettyref{prop:classification_mergings} implies that every $2$ in a column of
	$\pi_{(R,S)}$ appears above a $0$, and condition~(4) in \prettyref{prop:classification_mergings}
	implies that every $0$ in a row of $\pi_{(R,S)}$ appears to the right of a $2$. Hence, 
	$\pi_{(R,S)}$ is a plane partition with $m$ rows, $n$ columns and largest part at most $2$.
\end{proof}

An extensive illustration of this bijection can be found in \prettyref{app:bijection_chains}.

\subsection{The Number of Proper Mergings of Two Chains}
  \label{sec:counting_mergings_chains}
Having the bijection from the previous section in mind, it is now straight-forward to determine the 
number of proper mergings of two chains. Let us recall a classical result by MacMahon.
\begin{theorem}
  \label{thm:macmahon}
	Let $l,m,n\in\mathbb{N}$. The number $\pi(m,n,l)$ of plane partitions with $m$ rows, $n$ 
	columns and largest part at most $l$ is given by
	\begin{align}
	  \label{eq:macmahon}
		\pi(m,n,l)=\prod_{i=1}^{m}\prod_{j=1}^{n}\prod_{k=1}^{l}{\frac{i+j+k-1}{i+j+k-2}}.
	\end{align}
\end{theorem}
This result was first conjectured in \cite{macmahon97memoir} and later proven in 
\cite{macmahon16combinatory}*{Sections~XI~and~X}. The presented form can be derived from 
\cite{macdonald95symmetric}*{Example~13(b)}.

\begin{proof}[Proof of \prettyref{thm:enumeration_chains}]
	\prettyref{thm:bijection_chains} and \prettyref{thm:macmahon} imply
	\begin{align*}
		\left\lvert\mathfrak{C}_{m,n}^{\bullet}\right\rvert & = \pi(m,n,2)\\
		& = \prod_{i=1}^{m}\prod_{j=1}^{n}{\frac{i+j+1}{i+j-1}}\\
		& = \prod_{i=1}^{m}{\frac{i+m}{i}\cdot\frac{i+m+1}{i+1}}\\
		& = \frac{1}{m+n+1}\binom{m+n+1}{m+1}\binom{m+n+1}{m}.
	\end{align*}
\end{proof}

\begin{remark}
	Consider the \emph{Narayana numbers} (see \cite{stanley01enumerative}*{Exercise~6.36~a}), 
	defined by
	\begin{align}
	  \label{eq:narayana}
		\text{Nar}(\tilde{n},\tilde{m})=
		  \frac{1}{\tilde{n}}\binom{\tilde{n}}{\tilde{m}}\binom{\tilde{n}}{\tilde{m}-1},
	\end{align}
	for $\tilde{m},\tilde{n}\in\mathbb{N}$, with $\tilde{m}\leq\tilde{n}$. In view of 
	\prettyref{thm:bijection_chains}, we obtain
	\begin{align*}
		\left\lvert\mathfrak{C}_{m,n}^{\bullet}\right\rvert = \text{Nar}(m+n+1,m+1).
	\end{align*}
\end{remark}

\begin{remark}
	Let $\pi=(\pi_{i,j})_{1\leq i\leq m,1\leq j\leq n}$ and 
	$\sigma=(\sigma_{i,j})_{1\leq i\leq m,1\leq j\leq n}$ be plane partitions with $m$ rows
	and $n$ columns, and largest part $2$. Define a partial order $\leq$ on 
	$\text{PP}_{m,n}^{(2)}$ as
	\begin{align*}
		\pi\leq\sigma\quad\text{if and only if}\quad \pi_{i,j}\leq\sigma_{i,j},
	\end{align*}
	for all $1\leq i\leq m$ and $1\leq j\leq n$. Let $(R_{\pi},S_{\pi})$, and
	$(R_{\sigma},S_{\sigma})$ denote the proper mergings associated to $\pi$ respectively $\sigma$ in
	the sense of \prettyref{def:bijection_chains}. Suppose that
	$(R_{\pi},S_{\pi})\preceq(R_{\sigma},S_{\sigma})$, and hence by definition 
	$R_{\pi}\subseteq R_{\sigma}$, and $S_{\pi}\supseteq S_{\sigma}$. This implies that if
	$\pi_{i,j}=2$, then $\sigma_{i,j}=2$. If $\pi_{i,j}=1$, then $\sigma_{i,j}\in\{1,2\}$, and if
	$\pi_{i,j}=0$, then $\sigma_{i,j}\in\{0,1,2\}$. Hence, $\pi\leq\sigma$. This means that the
	bijection described in \prettyref{thm:bijection_chains} is indeed an isomorphism between the 
	lattices $(\mathfrak{C}_{m,n}^{\bullet},\preceq)$ and $(\text{PP}_{m,n}^{(2)},\leq)$.
\end{remark}

\begin{figure}
	\begin{tikzpicture}
	\def\x{2.5};
	\def\y{1.75};
	\def\r{.8};
	\def\rs{.8};
	\draw(3*\x,1*\y) node[circle,draw,scale=\r](v1){};
	\draw(2*\x,2*\y) node[circle,draw,scale=\r](v2){};
	\draw(4*\x,2*\y) node[circle,draw,scale=\r](v3){};
	\draw(1*\x,3*\y) node[circle,draw,scale=\r](v4){};
	\draw(2*\x,3*\y) node[circle,draw,scale=\r](v5){};	
	\draw(3*\x,3*\y) node[circle,draw,scale=\r](v6){};
	\draw(4*\x,3*\y) node[circle,draw,scale=\r](v7){};
	\draw(5*\x,3*\y) node[circle,draw,scale=\r](v8){};
	\draw(1*\x,4*\y) node[circle,draw,scale=\r](v9){};
	\draw(3*\x,4*\y) node[circle,draw,scale=\r](v10){};
	\draw(5*\x,4*\y) node[circle,draw,scale=\r](v11){};
	\draw(1*\x,5*\y) node[circle,draw,scale=\r](v12){};
	\draw(2*\x,5*\y) node[circle,draw,scale=\r](v13){};
	\draw(3*\x,5*\y) node[circle,draw,scale=\r](v14){};
	\draw(4*\x,5*\y) node[circle,draw,scale=\r](v15){};
	\draw(5*\x,5*\y) node[circle,draw,scale=\r](v16){};
	\draw(2*\x,6*\y) node[circle,draw,scale=\r](v17){};
	\draw(4*\x,6*\y) node[circle,draw,scale=\r](v18){};
	\draw(3*\x,7*\y) node[circle,draw,scale=\r](v19){};
	\draw(v1) -- (v2);
	\draw(v1) -- (v3);
	\draw(v2) -- (v4);
	\draw(v2) -- (v5);
	\draw(v2) -- (v6);
	\draw(v3) -- (v6);
	\draw(v3) -- (v7);
	\draw(v3) -- (v8);
	\draw(v4) -- (v9);
	\draw(v5) -- (v9);
	\draw(v5) -- (v10);
	\draw(v6) -- (v10);
	\draw(v7) -- (v10);
	\draw(v7) -- (v11);
	\draw(v8) -- (v11);
	\draw(v9) -- (v12);
	\draw(v9) -- (v13);
	\draw(v10) -- (v13);
	\draw(v10) -- (v14);
	\draw(v10) -- (v15);
	\draw(v11) -- (v15);
	\draw(v11) -- (v16);
	\draw(v12) -- (v17);
	\draw(v13) -- (v17);
	\draw(v14) -- (v17);
	\draw(v14) -- (v18);
	\draw(v15) -- (v18);
	\draw(v16) -- (v18);
	\draw(v17) -- (v19);
	\draw(v18) -- (v19);
	\draw(3.33*\x,1*\y) node{\uOne{0}{\rs}};
	\draw(1.67*\x,1.75*\y) node{\uTwo{0}{\rs}};
	\draw(4.33*\x,1.75*\y) node{\uThree{0}{\rs}};
	\draw(.67*\x,3*\y) node{\uFour{0}{\rs}};
	\draw(1.67*\x,3*\y) node{\uFive{0}{\rs}};
	\draw(3*\x,2.67*\y) node{\uSix{0}{\rs}};
	\draw(4.33*\x,3*\y) node{\uSeven{0}{\rs}};
	\draw(5.33*\x,3*\y) node{\uEight{0}{\rs}};
	\draw(.67*\x,4*\y) node{\uNine{0}{\rs}};
	\draw(3.33*\x,4*\y) node{\uTen{0}{\rs}};
	\draw(5.33*\x,4*\y) node{\reflectbox{\uNine{180}{\rs}}};
	\draw(.67*\x,5*\y) node{\uEight{180}{\rs}};
	\draw(1.67*\x,5*\y) node{\uSeven{180}{\rs}};
	\draw(3*\x,5.33*\y) node{\uSix{180}{\rs}};
	\draw(4.33*\x,5*\y) node{\uFive{180}{\rs}};
	\draw(5.33*\x,5*\y) node{\uFour{180}{\rs}};
	\draw(1.67*\x,6.25*\y) node{\uThree{180}{\rs}};
	\draw(4.33*\x,6.25*\y) node{\reflectbox{\uTwo{180}{\rs}}};
	\draw(2.67*\x,7*\y) node{\uOne{180}{\rs}};
\end{tikzpicture}
	\caption{The lattice of proper mergings of three $1$-chains.}
	\label{fig:mergings_three_chains}
\end{figure}
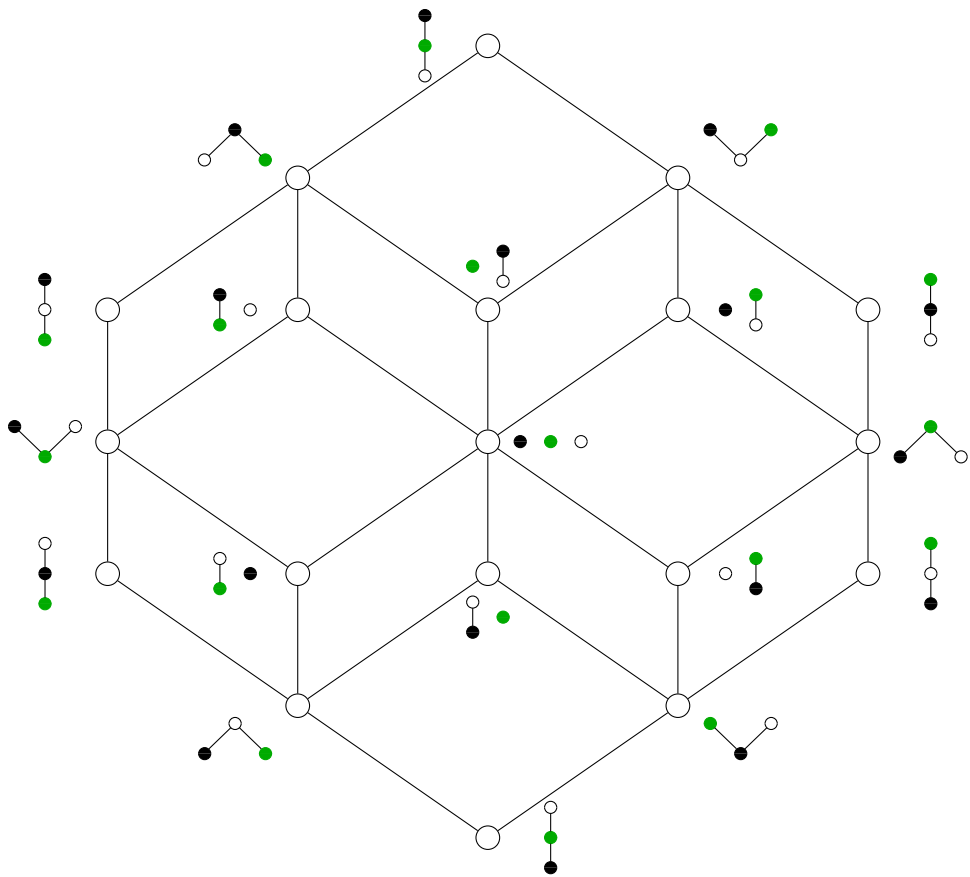

\begin{remark}
  \label{rem:generalized_bijection_chains}
	Christian Meschke proposed the following generalization of mergings of quasi-ordered sets: 
	let $T$ be a linearly ordered set, and let $(P_{t},\leftarrow_{t})_{t\in T}$ be a family of 
	quasi-ordered sets, indexed by $T$. Define $P=\bigcup_{t\in T}{P_{t}}$, and let 
	$R\subseteq P\times P$ be a relation on $P$. We abbreviate $R_{s,t}=R\cap(P_{s}\times P_{t})$.
	Then, $R$ is called \emph{merging of the $P_{t}$'s} if it is a quasi-ordered set on $P$ such 
	that $R_{t,t}$ yields $\leftarrow_{t}$ again. Moreover, $R$ is called \emph{proper} if for all
	$s<t$, we have $R_{s,t}\cap R_{t,s}^{-1}=\emptyset$. Let $\mathfrak{M}_{T}$ denote the set
	of all mergings of the $P_{t}$'s. We define a partial order $\sqsubseteq$ on 
	$\mathfrak{M}_{T}$ as 
	\begin{align*}
		R\sqsubseteq S\quad\text{if and only if}\quad
		\begin{cases}
		R_{s,t}\subseteq S_{s,t} & \quad\text{if}\;s<t,\quad\text{and}\\
		R_{s,t}\supseteq S_{s,t} & \quad\text{if}\;s>t,
		\end{cases}
	\end{align*}
	for all $R,S\in\mathfrak{M}_{T}$. Then, $(\mathfrak{M}_{T},\sqsubseteq)$ is again a lattice. 
	However, as we notice from \prettyref{fig:mergings_three_chains}, this lattice is in general no 
	longer distributive. Even more, up to now it is not clear, how to construct the formal 
	context which generates $(\mathfrak{M}_{T},\sqsubseteq)$ from the quasi-orders $\leftarrow_{t}$.
	
	We can now think of a generalization of the bijection described in 
	\prettyref{thm:bijection_chains} to proper mergings of more than two chains in the following 
	way: let $\cc_{1},\cc_{2},\ldots,\cc_{t}$ be chains, where for all $i\in\{1,2,\ldots,t\}$, the 
	chain $\cc_{i}$ has $n_{i}$ elements. Consider the standard unit vectors $e_{1},e_{2},\ldots,e_{t}$ 
	in $\mathbb{R}^{t}$, and label the points 
	$e_{j}-\tfrac{1}{2},2e_{j}-\tfrac{1}{2},\ldots,n_{j}e_{j}-\tfrac{1}{2}$ with the elements of the 
	chain $\cc_{j}$ for all $j\in\{1,2,\ldots,t\}$ in the obvious way. For each 
	$i,j\in\{1,2,\ldots,t\}$ with $i<j$, we can insert 
	a plane partition with largest part $\leq 2$ into the $(n_{i}\times n_{j})$-array, spanned by 
	the vectors $n_{i}e_{i}$ and $n_{j}e_{j}$, and call this an \emph{arrangement of $t$ plane 
	partitions}.
	
	For an illustration of this construction, we refer to \prettyref{fig:generalized_bijection}. 
	On the left of each figure, there is an arrangement of three plane partitions with one row and one
	column, together with the labeled coordinate axes. In the middle, the three plane 
	partitions are written next to each other and on the right, there is the merging of three 
	$1$-chains which is induced by these plane partitions in the spirit of 
	\prettyref{def:bijection_chains}. We notice that \prettyref{fig:generalized_proper} shows a 
	proper merging of the three $1$-chains, while \prettyref{fig:generalized_improper} does not. 
	See \prettyref{app:generalized_bijection_chains} for an extensive illustration of the case of 
	proper mergings of three $1$-chains. In this appendix, we also notice that some arrangements 
	of plane partitions yield the same mergings. 

	If this construction can indeed be used as a generalization of \prettyref{thm:bijection_chains} 
	should be investigated in a subsequent article.
\end{remark}

\begin{figure}
	\subfigure[An arrangement of three plane partitions, which yields a proper merging of three 
	  $1$-chains.]{\label{fig:generalized_proper}
		\begin{tikzpicture}
			\draw(0,0) node{\exampleFiveSolidPartition{$0$}{$0$}{$1$}};
			\draw(1.75,-.25) node{$\rightarrow$};
			\draw(3,0) node{\exampleFiveSmallPartition{$a$}{$b$}{$0$}};
			\draw(4.75,0) node{\exampleFiveSmallPartition{$b$}{$c$}{$0$}};
			\draw(6.5,0) node{\exampleFiveSmallPartition{$c$}{$a$}{$1$}};
			\draw(8,-.25) node{$\rightarrow$};
			\draw(9.5,0) node{\exampleFivePosetOne};
		\end{tikzpicture}
	}
	\subfigure[An arrangement of three plane partitions, which does not yield a proper merging of 
	  three $1$-chains.]{\label{fig:generalized_improper}
		\begin{tikzpicture}
			\draw(0,0) node{\exampleFiveSolidPartition{$0$}{$0$}{$0$}};
			\draw(1.75,-.25) node{$\rightarrow$};
			\draw(3,0) node{\exampleFiveSmallPartition{$a$}{$b$}{$0$}};
			\draw(4.75,0) node{\exampleFiveSmallPartition{$b$}{$c$}{$0$}};
			\draw(6.5,0) node{\exampleFiveSmallPartition{$c$}{$a$}{$0$}};
			\draw(8,-.25) node{$\rightarrow$};
			\draw(9.5,-.25) node{\exampleFivePosetTwo};
		\end{tikzpicture}
	}
	\caption{Two examples of posets induced by an arrangement of three plane partitions.}
	\label{fig:generalized_bijection}
\end{figure}
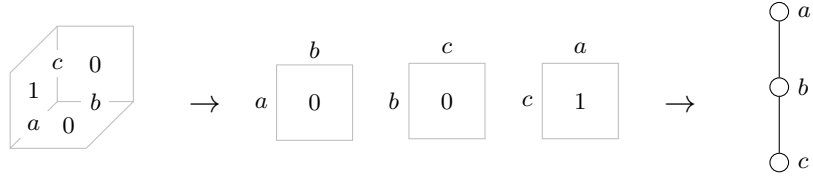
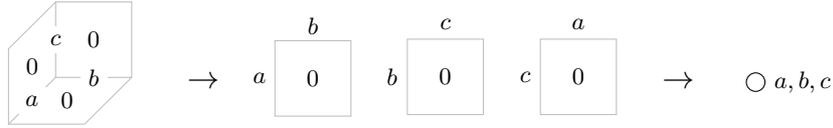

\subsection{Counting Galois Connections between Chains}
  \label{sec:galois_connections_chains}
In this section, we describe how we can exploit the bijection given in 
\prettyref{def:bijection_chains} to allow for counting Galois connections between 
two chains. Let us first recall the definitions. A \emph{Galois connection} between two posets 
$(P,\leq_{P})$ and $(Q,\leq_{Q})$ is a pair $(\varphi,\psi)$ of maps
\begin{align*}
	\varphi: P\to Q\quad\text{and}\quad\psi:Q\to P,
\end{align*}
satisfying 
\begin{align}
	& p_{1} \leq_{P} p_{2} \quad\text{implies}\quad\varphi p_{1}\geq_{Q}\varphi p_{2},\\
	& q_{1} \leq_{Q} q_{2} \quad\text{implies}\quad\psi q_{1}\geq_{P}\psi q_{2},\\
	& p\leq_{P}\psi\varphi p, \quad\text{and}\quad
	 q\leq_{Q}\varphi\psi q,
\end{align}
for all $p,p_{1},p_{2}\in P$ and $q,q_{1},q_{2}\in Q$. Now, let $(P,\leq_{P})\cong\CL(\KK_{1})$ and 
$(Q,\leq_{Q})\cong\CL(\KK_{2})$ be concept lattices, where $\KK_{1}=(G,M,I)$ and $\KK_{2}=(H,N,J)$ 
are the corresponding formal contexts. In this particular case, \prettyref{thm:galois_connections} 
below states that each Galois connection from $\CL(\KK_{1})$ to $\CL(\KK_{2})$ corresponds to a dual 
bond from $\KK_{1}$ to $\KK_{2}$. A relation $R\subseteq G\times H$, is called 
\emph{dual bond from $\KK_{1}$ to $\KK_{2}$} if for every $g\in G$, the set $\{g\}^{R}$ is an extent 
of $\KK_{2}$ and for every $h\in H$, the set $\{h\}^{R}$ is an extent of $\KK_{1}$. In other words, 
$R$ is a dual bond from $\KK_{1}$ to $\KK_{2}$ if and only if $R$ is a bond from $\KK_{1}$ to the 
dual\footnote{Let $\KK=(G,M,I)$ be a formal context. The dual context $\KK^{d}$ of $\KK$ is given by
$(M,G,I^{-1})$ and satisfies $\CL(\KK^{d})\cong\CL(\KK)^{d}$, where $\CL(\KK)^{d}$ is the 
(order-theoretic) dual of the lattice $\CL(\KK)$.} context $\KK_{2}^{d}$. 

\begin{theorem}[\cite{ganter99formal}*{Theorem~53}]
  \label{thm:galois_connections}
	Let $(G,M,I)$ and $(H,N,J)$ be formal contexts. For every dual bond $R\subseteq G\times H$, the 
	maps
	\begin{align*}
		\varphi_{R}\bigl(X,X^{I}\bigr)=\bigl(X^{R},X^{RJ}\bigr),\quad\text{and}\quad
		  \psi_{R}\bigl(Y,Y^{J}\bigr)=\bigl(Y^{R},Y^{RI}\bigr),
	\end{align*}
	where $X$ and $Y$ are extents of $(G,M,I)$ respectively $(H,N,J)$, form a Galois connection 
	between $\CL(G,M,I)$ and $\CL(H,N,J)$. Moreover, every Galois connection $(\varphi,\psi)$ 
	induces a dual bond from $(G,M,I)$ to $(H,N,J)$ by 
	\begin{align*}
		R_{(\varphi,\psi)}=\bigl\{(g,h)\mid\gamma g\leq\psi\gamma h\bigr\}=
		  \bigl\{(g,h)\mid\gamma h\leq\varphi\gamma g\bigr\},
	\end{align*}
	where $\gamma$ is the map defined in \eqref{eq:object}. We have
	\begin{align*}
		\varphi_{R_{(\varphi,\psi)}}=\varphi,\quad\psi_{R_{(\varphi,\psi)}}=\psi,\quad
		  \text{and}\quad R_{(\varphi_{R},\psi_{R})}=R.
	\end{align*}
\end{theorem}
Let $C_{1}=\{a_{1},a_{2},\ldots,a_{m}\}$ and $C_{2}=\{b_{1},b_{2},\ldots,b_{n}\}$ be sets, and 
consider the corresponding chains $\cc_{1}=(C_{1},\leq_{1})$ and $\cc_{2}=(C_{2},\leq_{2})$, where 
the order relations are given by the indices of the corresponding sets. We can easily deduce from the 
reasoning in \prettyref{sec:intents_extents} that a relation $R\subseteq C_{1}\times C_{2}$ is a dual 
bond from $\KK(\cc_{1})$ to $\KK(\cc_{2})$ if and only if it satisfies
\begin{align*}
	\{a\}^{R} & =\{b_{1},b_{2},\ldots,b_{i}\},\quad\text{and}\\
	\{b\}^{R} & =\{a_{1},a_{2},\ldots,a_{j}\},
\end{align*}
for every $a\in C_{1}, b\in C_{2}$, and some $i\in\{0,1,\ldots,n\}$, and some $j\in\{0,1,\ldots,m\}$. 
(Again, the cases $i=0$ and $j=0$ are to be interpreted as the empty set.) 

We also noticed in \prettyref{sec:intents_extents} that an $n$-chain $\cc$ is isomorphic to the
concept lattice of the formal context $\CC(\cc')$, for some $(n-1)$-chain $\cc'$. Hence, if 
$\cc_{1}$ and $\cc_{2}$ are $m$- respectively $n$-chains, and $\cc'_{1}$ and $\cc'_{2}$ are $(m-1)$- 
respectively $(n-1)$-chains, we can interpret each dual bond from $\KK(\cc_{1})$ to $\KK(\cc_{2})$ as 
a dual bond from $\CC(\cc'_{1})$ to $\CC(\cc'_{2})$. This observation is crucial for the proof of the 
following proposition.

\begin{proposition}
  \label{prop:galois_connections_chains}
	Let $m,n\in\mathbb{N}$. The number of Galois connections between an $m$-chain and an 
	$n$-chain is $\tbinom{m+n-2}{m-1}$.
\end{proposition}
\begin{proof}
	Let $\cc_{1}$ be an $m$-chain and let $\cc_{2}$ be an $n$-chain. Let $\cc'_{1}$ be an 
	$(m-1)$-chain, and let $\cc'_{2}$ be an $(n-1)$-chain. Note that
	\begin{align*}
		\CL\bigl(\KK(\cc_{1})\bigr)\cong\CL\bigl(\CC(\cc'_{1})\bigr),\quad\text{and}\quad
		\CL\bigl(\KK(\cc_{2})\bigr)\cong\CL\bigl(\CC(\cc'_{2})\bigr).
	\end{align*}

	Since chains are self-dual, the set of dual bonds from $\KK(\cc_{1})$ to $\KK(\cc_{2})$ is 
	in bijection with the set of bonds from $\KK(\cc_{1})$ to $\KK(\cc_{2})$. Moreover, it 
	follows immediately from \prettyref{prop:classification_mergings} and the reasoning above that 
	$(\emptyset,S)$ is a proper merging of $\cc'_{1}$ and $\cc'_{2}$ if and only if $S$ is a bond 
	from $\KK(\cc_{1})$ to $\KK(\cc_{2})$. (Note that every binary relation $S$ satisfies 
	$S\circ\emptyset=\emptyset=\emptyset\circ S$. Moreover, $\emptyset$ is a 
	bond between two formal contexts $\KK_{1}$ and $\KK_{2}$ if and only if $\KK_{2}$ does not 
	contain a full row and $\KK_{1}$ does not contain a full column. Since neither 
	$\CC(\cc_{1}')$ nor $\CC(\cc_{2}')$ contain full rows or full columns, the conditions in 
	\prettyref{prop:classification_mergings} for $(\emptyset,S)$ to be a proper merging of 
	$\cc'_{1}$ and $\cc'_{2}$ reduce to $S$ being a bond from $\CC(\cc'_{1})$ to $\CC(\cc'_{2})$. 
	The latter is equivalent to $S$ being a bond from $\KK(\cc_{1})$ to $\KK(\cc_{2})$, when 
	identifying $S$ with the corresponding relation derived from the isomorphisms between 
	$\CC(\cc'_{1})$ and $\KK(\cc_{1})$, respectively $\CC(\cc'_{2})$ and $\KK(\cc_{2})$.) Thus, 
	every Galois connection between $\cc_{1}$ and $\cc_{2}$ corresponds by 
	\prettyref{thm:galois_connections} and the previous reasoning to a proper merging of 
	$\cc'_{1}$ and $\cc'_{2}$, which is of the form $(\emptyset,\cdot)$.
	
	Let us make this correspondence more explicit. By the bijection given in 
	\prettyref{def:bijection_chains}, it is clear that a proper merging of $\cc'_{1}$ and 
	$\cc'_{2}$, which is of the form $(\emptyset,\cdot)$, corresponds to a plane partition with 
	$m-1$ rows, $n-1$ columns and largest part at most $1$. Let $\pi$ be such a plane partition, 
	and let $C'_{1}=\{a_{1},a_{2},\ldots,a_{m-1}\}$ be the ground set of $\cc'_{1}$ and let 
	$C'_{2}=\{b_{1},b_{2},\ldots,b_{n-1}\}$ be the ground set of $\cc'_{2}$. Let 
	$S_{\pi}\subseteq C'_{2}\times C'_{1}$ be the relation given in 
	\prettyref{def:bijection_chains}. Define a relation 
	$T_{\pi}\subseteq C'_{2}\times C'_{1}$ as
	\begin{align}
		b_{j}\;T_{\pi}\; a_{i}\quad\text{if and only if}\quad b_{j}\;S_{\pi}\;a_{n-i+1},
	\end{align}
	for $1\leq i\leq m-1$, and $1\leq j\leq n-1$. Thus, $T_{\pi}$ (as a cross-table) corresponds
	to a horizontal reflection of $S_{\pi}$ (as a cross-table). It is now immediate from the
	construction that the rows of $T_{\pi}$ are of the form $\{a_{1},a_{2},\ldots,a_{j}\}$ for some
	$j\in\{0,1,\ldots,m-1\}$, and the columns of $T_{\pi}$ are of the form 
	$\{b_{1},b_{2},\ldots,b_{i}\}$ for some $i\in\{0,1,\ldots,n-1\}$. Since $S_{\pi}$ is a bond 
	between $\CC(\cc'_{2})$ and $\CC(\cc'_{1})$, we can conclude that $T_{\pi}$ is a dual bond 
	between $\CC(\cc'_{2})$ and $\CC(\cc'_{1})$. By symmetry, $T_{\pi}$ induces a Galois connection 
	between $\cc_{1}$ and $\cc_{2}$. 

	The number of plane partitions with $m-1$ rows, $n-1$ columns and largest part at most $1$ 
	can be computed from \prettyref{thm:macmahon}, and it turns out to be $\tbinom{m+n-2}{m-1}$.
\end{proof}

\prettyref{fig:example_galois_connection} shows an example of a Galois connection between a $5$-chain 
and a $7$-chain arising from a plane partion with $6$ rows and $4$ columns and largest part $1$. An 
extensive illustration of the bijection described in the proof of 
\prettyref{prop:galois_connections_chains} can be found in \prettyref{app:bijection_galois_bonds}.

\begin{figure}
	\centering
	\begin{tikzpicture}
		\draw(0,6.9) node(v1){
		  \subfigure[A plane partition.]{
		    \exampleThreePartition
		  }};
		\draw(6,6) node(v2){
		  \subfigure[The corresponding relations $R$ and $S$.]{
		    \begin{tikzpicture}
			    \draw(0,0) node{\exampleThreeSTable};
			    \draw(0,3) node{\exampleThreeRTable};
		    \end{tikzpicture}
		  }};
		\draw(.25,.4) node(v3){
		  \subfigure[The corresponding proper merging.]{
		    \exampleThreePoset
		  }};
		\draw(6,-1) node(v4){
		  \subfigure[The corresponding relation $T$.]{
		    \begin{tikzpicture}
			    \draw(0,0) node{\exampleThreeTTable};
		    \end{tikzpicture}
		  }};
		\draw(3.5,-7) node(v5){
		  \subfigure[The corresponding Galois connection.]{
		    \exampleThreeGalois
		  }};
	\end{tikzpicture}
	\caption{A plane partition, the induced proper merging of a $6$-chain and a $4$-chain, the 
	  corresponding dual bond, and the induced Galois connection between a $5$-chain and a
	  $7$-chain.}
	\label{fig:example_galois_connection}
\end{figure}
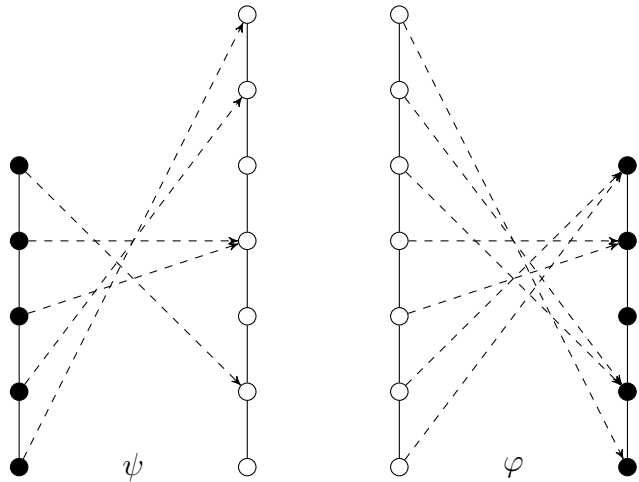

\section{Proper Mergings of two Antichains}
  \label{sec:antichains}
In this section, we investigate the number  of the proper mergings of two antichains. In particular, 
we prove the following theorem.

\begin{theorem}
  \label{thm:enumeration_antichains}
	Let $\mathfrak{A}_{m,n}^{\bullet}$ denote the set of proper mergings of an $m$-antichain 
	and an $n$-antichain. Then,
	\begin{align*}
		\left\lvert\mathfrak{A}_{m,n}^{\bullet}\right\rvert = 
		  \sum_{n_{1}+m_{1}+k_{1}=m}{\binom{m}{n_{1},m_{1},k_{1}}(-1)^{k_{1}}
		  \Bigl(2^{n_{1}}+2^{m_{1}}-1\Bigr)^{n}}.
	\end{align*}
\end{theorem}

Let $\ac_{1}$ and $\ac_{2}$ be antichains. It is obvious that the Hasse diagram of a proper merging of 
$\ac_{1}$ and $\ac_{2}$ can be regarded as a (not necessarily connected) bipartite 
graph. \prettyref{fig:mergings_two_antichains} shows the lattice of proper mergings of two 
$2$-antichains, where the nodes are labeled by the corresponding proper mergings. 
\begin{figure}
	\centering
	\include{mergings_two_antichains}
	\caption{The lattice of proper mergings of two $2$-antichains.}
	\label{fig:mergings_two_antichains}
\end{figure}
In order to prove \prettyref{thm:enumeration_antichains}, we construct the generating function
of proper mergings of two antichains. 

Let $B(x,y)$ denote the bivariate exponential generating function of bipartite graphs. The vertex set
of a bipartite graph can be partitioned into two sets $V_{1}$ and $V_{2}$. Say that the variable $x$ 
counts the cardinality of $V_{1}$ and the variable $y$ counts the cardinality of $V_{2}$. Let $b(m,n)$
denote the number of bipartite graphs with vertex set $V=V_{1}\cup V_{2}$, and 
$\lvert V_{1}\rvert=m, \lvert V_{2}\rvert=n$. Clearly, then $b(m,n)=2^{mn}$, and we find 
\begin{align}
	\label{eq:bipartite}
	  B(x,y) & =\sum_{n\geq 0}{\sum_{m\geq 0}{b(m,n)\frac{x^{n}}{n!}\cdot\frac{y^{m}}{m!}}}\\
	  \nonumber & =\sum_{n\geq 0}{\sum_{m\geq 0}{2^{mn}\frac{x^{n}}{n!}\cdot\frac{y^{m}}{m!}}}.
\end{align}
Let $B_{c}(x,y)$ denote the bivariate exponential generating function for \emph{connected} bipartite 
graphs. Since every bipartite graph can be seen as a collection of connected bipartite graphs, 
we obtain
\begin{align}
	\label{eq:connected_bipartite}
	  B(x,y)=\exp\bigl(B_{c}(x,y)\bigr).
\end{align}
See for instance \cite{wilf06generating}*{Chapter~3} for an explanation of this equality. In 
particular, this correspondence is a bivariate exponential generating function version of
\cite{wilf06generating}*{Theorem~3.4.1}. Now we are able to prove 
\prettyref{thm:enumeration_antichains}.

\begin{proof}[Proof of \prettyref{thm:enumeration_antichains}]
	Let $(R,S)$ be a proper merging of an $m$-antichain $\ac_{1}$ and an $n$-antichain 
	$\ac_{2}$. Denote by $\{\beta_{1},\beta_{2},\ldots,\beta_{k}\}$ the set of connected 
	components of the Hasse diagram of $(R,S)$ (considered as a graph). Clearly, each 
	$\beta_{i}$ is a connected bipartite graph. Without loss of generality, we can assume
	that the vertices of $\beta_{i}$ which belong to $\ac_{1}$ are below the vertices of
	$\beta_{i}$ which belong to $\ac_{2}$. Then, we can flip the graph in such a way that 
	the vertices of $\beta_{i}$ which belong to $\ac_{1}$ are above the vertices of $\beta_{i}$
	which belong to $\ac_{2}$, and edges are preserved. This procedure yields another 
	connected bipartite graph, say $\beta_{i}^{d}$. For every $i\in\{1,2,\ldots,k\}$, it is 
	clear that the set 
	$\{\beta_{1},\beta_{2},\ldots,\beta_{i-1},\beta_{i}^{d},\beta_{i+1},\ldots,\beta_{k}\}$ is 
	the set of connected components of the Hasse diagram of another proper merging, say $(R,S)^{(i)}$, 
	of $\ac_{1}$ and $\ac_{2}$. It is immediate that $(R,S)$ and $(R,S)^{(i)}$ are different 
	proper mergings of $\ac_{1}$ and $\ac_{2}$ if and only if $\beta_{i}$ has more than one 
	vertex. (See \prettyref{fig:flips} for an illustration.)

	\begin{figure}[h]
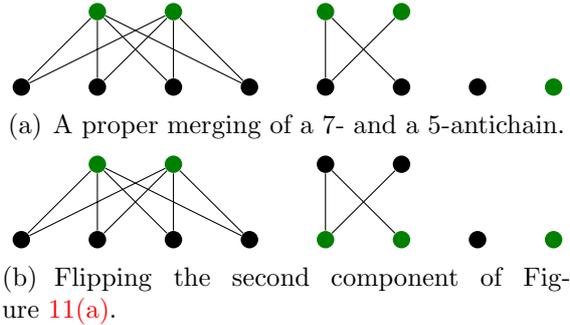

		\centering
		\subfigure[A proper merging of a $7$- and a $5$-antichain.]{\label{fig:flips_a}
			\connectedComponentsNormal
		}
		\subfigure[Flipping the second component of \prettyref{fig:flips_a}.]{\label{fig:flips_b}
			\connectedComponentsFlipped
		}
		\caption{Flipping a connected component of a proper merging of a $7$-antichain and a 
		  $5$-antichain.}
		\label{fig:flips}
	\end{figure}

	Let $G(x,y)$ denote the bivariate exponential generating function of proper mergings 
	of two antichains. The previous reasoning implies that every proper merging of two antichains 
	can be regarded as a collection of connected bipartite graphs 
	$\beta_{1},\beta_{2},\ldots,\beta_{k}$. Moreover, each connected component $\beta_{i}$ can 
	appear in two positions, namely $\beta_{i}$ and $\beta_{i}^{d}$, unless it has only one 
	vertex. Again, in the spirit of \cite{wilf06generating}*{Theorem~3.4.1}, we can write this 
	down as 
	\begin{align}
		\label{eq:correspondence}
		  G(x,y) = \exp\bigl(2\cdot B_{c}(x,y)-x-y)\bigr).
	\end{align}

	Putting \eqref{eq:bipartite}, \eqref{eq:connected_bipartite} and 
	\eqref{eq:correspondence} together, we obtain
	{\allowdisplaybreaks
	\begin{align*}
	      G(x,y) & = \exp\bigl(2\cdot\log B(x,y)-x-y\bigr)\\
	      & = B(x,y)^{2}\cdot\sum_{k_{1}\geq 0}{\frac{(-x)^{k_{1}}}{k_{1}!}}\cdot
		\sum_{k_{2}\geq 0}{\frac{(-y)^{k_{2}}}{k_{2}!}}\\
	      & = \left(\sum_{n\geq 0}{\sum_{m\geq 0}{2^{mn}\frac{x^{n}}{n!}\cdot
		\frac{y^{m}}{m!}}}\right)^{2}\cdot\sum_{k_{1}\geq 0}{\frac{(-x)^{k_{1}}}{k_{1}!}}
		\cdot\sum_{k_{2}\geq 0}{\frac{(-y)^{k_{2}}}{k_{2}!}}\\
	      & = \sum_{n_{1}\geq 0}{\sum_{n_{2}\geq 0}{2^{n_{1}n_{2}}\frac{x^{n_{1}}}{n_{1}!}
		\cdot\frac{y^{n_{2}}}{n_{2}!}}}
		\cdot\sum_{m_{1}\geq 0}{\sum_{m_{2}\geq 0}{2^{m_{1}m_{2}}
		\frac{x^{m_{1}}}{m_{1}!}\cdot\frac{y^{m_{2}}}{m_{2}!}}}\\
	      & \kern1cm\times\sum_{k_{1}\geq 0}{\frac{(-x)^{k_{1}}}{k_{1}!}}
		\cdot\sum_{k_{2}\geq 0}{\frac{(-y)^{k_{2}}}{k_{2}!}}\\
	      & = \sum_{\begin{subarray}{l}n_{1}\geq 0,\;n_{2}\geq 0,\\m_{1}\geq 0,\;m_{2}\geq 0,\\
		k_{1}\geq 0,\;k_{2}\geq 0\end{subarray}}
		{2^{n_{1}n_{2}+m_{1}m_{2}}(-1)^{k_{1}}(-1)^{k_{2}}}
		\cdot{\frac{x^{n_{1}+m_{1}+k_{1}}}{n_{1}!\,m_{1}!\,k_{1}!}
		\cdot\frac{y^{n_{2}+m_{2}+k_{2}}}{n_{2}!\,m_{2}!\,k_{2}!}}.
	\end{align*}}%

	The number of proper mergings of an $m$- and an $n$-antichain is now given by the coefficient 
	of $\tfrac{x^{m}y^{n}}{m!\,n!}$ in $G(x,y)$. Hence,
	{\allowdisplaybreaks
	\begin{align*}
		\left\lvert\mathfrak{A}_{m,n}^{\bullet}\right\rvert & = 
		  \left\langle\tfrac{x^{m}y^{n}}{m!\,n!}\right\rangle G(x,y)\\
		& = m!\,n!\sum_{\begin{subarray}{l}n_{1}+m_{1}+k_{1}=m,\\ n_{2}+m_{2}+k_{2}=n
		  \end{subarray}}
		  {\frac{2^{n_{1}n_{2}}2^{m_{1}m_{2}}(-1)^{k_{1}}(-1)^{k_{2}}}
		  {n_{1}!\,n_{2}!\,m_{1}!\,m_{2}!\,k_{1}!\,k_{2}!}}\\
		& = \sum_{n_{1}+m_{1}+k_{1}=m}{\binom{m}{n_{1},m_{1},k_{1}}(-1)^{k_{1}}}\\
		& \kern1cm\times\sum_{n_{2}+m_{2}+k_{2}=n}{\binom{n}{n_{2},m_{2},k_{2}}
		  \bigl(2^{n_{1}}\bigr)^{n_{2}}\bigl(2^{m_{1}}\bigr)^{m_{2}}(-1)^{k_{1}}}\\
		& = \sum_{n_{1}+m_{1}+k_{1}=m}{\binom{m}{n_{1},m_{1},k_{1}}(-1)^{k_{1}}
		  \Bigl(2^{n_{1}}+2^{m_{1}}-1\Bigr)^{n}}.
	\end{align*}}%
\end{proof}

\section{Proper Mergings of an Antichain and a Chain}
  \label{sec:antichain_chain}
\begin{figure}
	\centering
	\begin{tikzpicture}
	\def\x{1.63};
	\def\y{2.2};
	\def\r{.8};
	\def\rn{.8};
	\draw(3.5*\x,1*\y) node[circle,draw,scale=\r](v1){};
	\draw(2.7*\x,2*\y) node[circle,draw,scale=\r](v2){};
	\draw(4.3*\x,2*\y) node[circle,draw,scale=\r](v3){};
	\draw(2*\x,3*\y) node[circle,draw,scale=\r](v4){};
	\draw(3.5*\x,3*\y) node[circle,draw,scale=\r](v5){};
	\draw(5*\x,3*\y) node[circle,draw,scale=\r](v6){};
	\draw(1.5*\x,4*\y) node[circle,draw,scale=\r](v7){};
	\draw(2.7*\x,4*\y) node[circle,draw,scale=\r](v8){};
	\draw(4.3*\x,4*\y) node[circle,draw,scale=\r](v9){};
	\draw(5.5*\x,4*\y) node[circle,draw,scale=\r](v10){};
	\draw(.5*\x,5.5*\y) node[circle,draw,scale=\r](v11){};
	\draw(1.5*\x,5.5*\y) node[circle,draw,scale=\r](v12){};
	\draw(2.7*\x,5.5*\y) node[circle,draw,scale=\r](v13){};
	\draw(4.3*\x,5.5*\y) node[circle,draw,scale=\r](v14){};
	\draw(5.5*\x,5.5*\y) node[circle,draw,scale=\r](v15){};
	\draw(6.5*\x,5.5*\y) node[circle,draw,scale=\r](v16){};
	\draw(1.5*\x,7*\y) node[circle,draw,scale=\r](v17){};
	\draw(2.7*\x,7*\y) node[circle,draw,scale=\r](v18){};
	\draw(4.3*\x,7*\y) node[circle,draw,scale=\r](v19){};
	\draw(5.5*\x,7*\y) node[circle,draw,scale=\r](v20){};
	\draw(2*\x,8*\y) node[circle,draw,scale=\r](v21){};
	\draw(3.5*\x,8*\y) node[circle,draw,scale=\r](v22){};
	\draw(5*\x,8*\y) node[circle,draw,scale=\r](v23){};
	\draw(2.7*\x,9*\y) node[circle,draw,scale=\r](v24){};
	\draw(4.3*\x,9*\y) node[circle,draw,scale=\r](v25){};
	\draw(3.5*\x,10*\y) node[circle,draw,scale=\r](v26){};
	\draw(v1) -- (v2);
	\draw(v1) -- (v3);
	\draw(v2) -- (v4);
	\draw(v2) -- (v5);
	\draw(v3) -- (v5);
	\draw(v3) -- (v6);
	\draw(v4) -- (v8);
	\draw(v5) -- (v7);
	\draw(v5) -- (v8);
	\draw(v5) -- (v9);
	\draw(v5) -- (v10);
	\draw(v6) -- (v9);
	\draw(v7) -- (v11);
	\draw(v7) -- (v13);
	\draw(v7) -- (v15);
	\draw(v8) -- (v11);
	\draw(v8) -- (v12);
	\draw(v8) -- (v14);
	\draw(v9) -- (v14);
	\draw(v9) -- (v15);
	\draw(v9) -- (v16);
	\draw(v10) -- (v12);
	\draw(v10) -- (v13);
	\draw(v10) -- (v16);
	\draw(v11) -- (v17);
	\draw(v11) -- (v18);
	\draw(v12) -- (v17);
	\draw(v12) -- (v19);
	\draw(v13) -- (v17);
	\draw(v13) -- (v20);
	\draw(v14) -- (v18);
	\draw(v14) -- (v19);
	\draw(v15) -- (v18);
	\draw(v15) -- (v20);
	\draw(v16) -- (v19);
	\draw(v16) -- (v20);
	\draw(v17) -- (v22);
	\draw(v18) -- (v21);
	\draw(v18) -- (v22);
	\draw(v19) -- (v22);
	\draw(v19) -- (v23);
	\draw(v20) -- (v22);
	\draw(v21) -- (v24);
	\draw(v22) -- (v24);
	\draw(v22) -- (v25);
	\draw(v23) -- (v25);
	\draw(v24) -- (v26);
	\draw(v25) -- (v26);
	\draw(4*\x,1.1*\y) node{\nOne{180}{\rn}};
	\draw(2.4*\x,2*\y) node{\nTwo{180}{\rn}};
	\draw(4.6*\x,2*\y) node{\nThree{180}{\rn}};
	\draw(1.7*\x,3*\y) node{\nFour{180}{\rn}};
	\draw(3.5*\x,2.65*\y) node{\nFive{180}{\rn}};
	\draw(5.3*\x,3*\y) node{\nSix{180}{\rn}};
	\draw(1.15*\x,4*\y) node{\nSeven{180}{\rn}};
	\draw(2.7*\x,3.65*\y) node{\nEight{180}{\rn}};
	\draw(4.3*\x,3.65*\y) node{\nNine{180}{\rn}};
	\draw(5.85*\x,4*\y) node{\nTen{180}{\rn}};
	\draw(.15*\x,5.5*\y) node{\nEleven{0}{\rn}};
	\draw(1.15*\x,5.5*\y) node{\nTwelve{180}{\rn}};
	\draw(2.35*\x,5.5*\y) node{\nThirteen{0}{\rn}};
	\draw(4.65*\x,5.5*\y) node{\nFourteen{0}{\rn}};
	\draw(5.85*\x,5.5*\y) node{\nTwelve{0}{\rn}};
	\draw(6.85*\x,5.5*\y) node{\nFifteen{0}{\rn}};
	\draw(1.15*\x,7*\y) node{\nSeven{0}{\rn}};
	\draw(2.7*\x,7.35*\y) node{\nEight{0}{\rn}};
	\draw(4.3*\x,7.35*\y) node{\nNine{0}{\rn}};
	\draw(5.85*\x,7*\y) node{\nTen{0}{\rn}};
	\draw(1.7*\x,8*\y) node{\nSix{0}{\rn}};
	\draw(3.5*\x,8.35*\y) node{\nFive{0}{\rn}};
	\draw(5.3*\x,8*\y) node{\nFour{0}{\rn}};
	\draw(4.6*\x,9*\y) node{\nTwo{0}{\rn}};
	\draw(2.4*\x,9*\y) node{\nThree{0}{\rn}};
	\draw(3*\x,9.9*\y) node{\nOne{0}{\rn}};
\end{tikzpicture}
	\caption{The lattice of proper mergings of a $2$-antichain and a $2$-chain.}
	\label{fig:mergings_antichain_chain}
\end{figure}
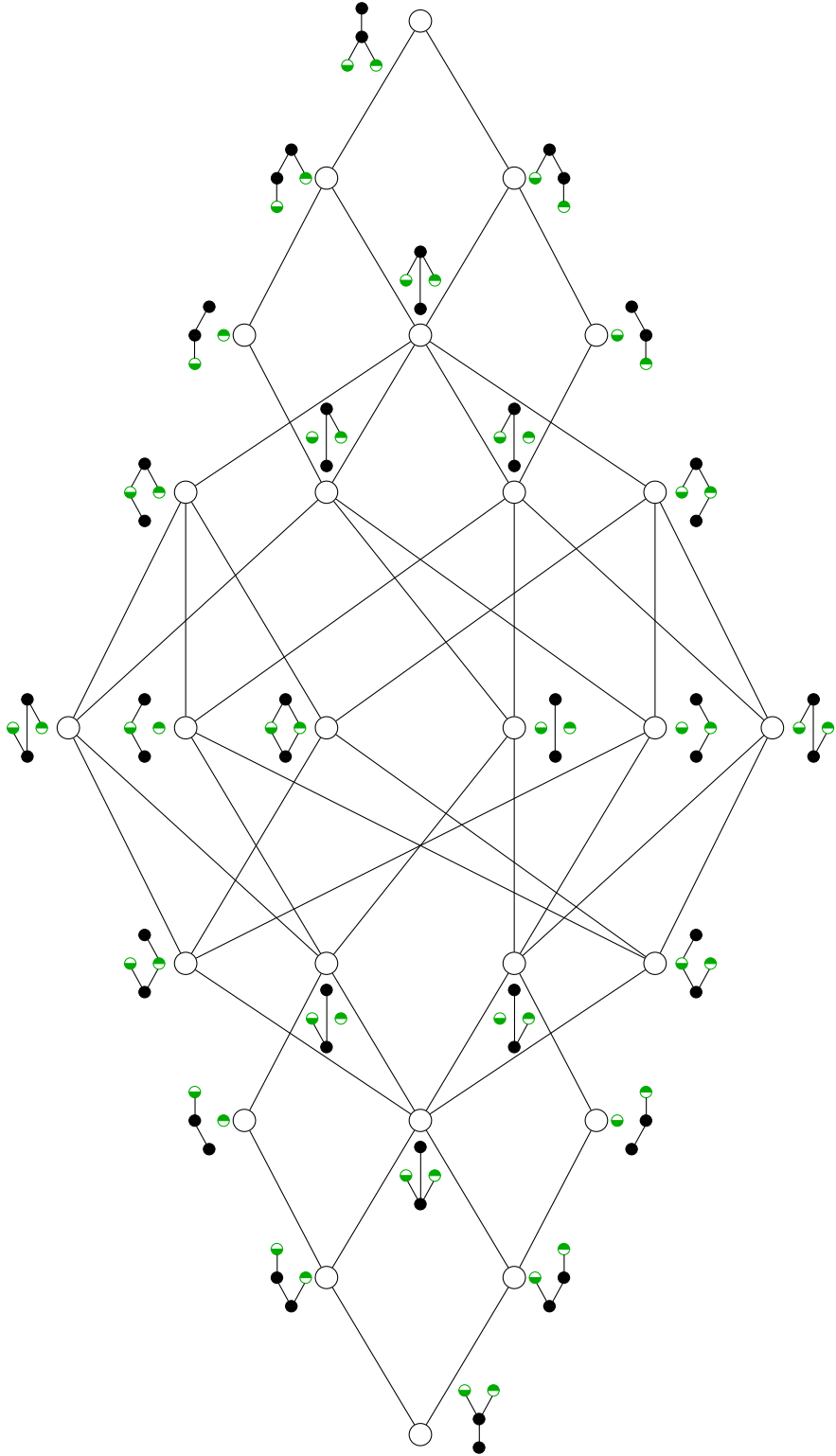

In this section, we investigate the family of proper mergings of an antichain and a chain. In
particular, we give a bijective proof of the following theorem.
\begin{theorem}
  \label{thm:enumeration_antichain_chain}
	Let $m,n\in\mathbb{N}$, and let $\mathfrak{A\!C}_{m,n}^{\bullet}$ denote the set of 
	proper mergings of an $m$-antichain and an $n$-chain. Then,
	\begin{align}
	  \label{eq:formula_antichain_chain}
		\left\lvert\mathfrak{A\!C}_{m,n}^{\bullet}\right\rvert
		  =\sum_{i=1}^{n+1}{\Bigl((n+2-i)^{m}-(n+1-i)^{m}\Bigr)i^{m}}.
	\end{align}
\end{theorem}
\begin{remark}
	We notice that in the case $m=0$, the equation \eqref{eq:formula_antichain_chain} contains
	a term of the form ``$0^{0}$'' which is per se undefined. Since there exists exactly one 
	proper merging of an empty antichain and some chain (namely the chain itself), it is reasonable
	to define the term ``$0^{0}$'' as being equal to $0$. This harmonizes well with 
	\prettyref{thm:bijection_antichain_chain} below, since there is exactly one monotone coloring
	of an empty graph.
\end{remark}

\prettyref{fig:mergings_antichain_chain} shows the lattice of proper mergings of a $2$-antichain and 
a $2$-chain, where the nodes are labeled by the corresponding proper mergings. Computer experiments 
show that the number of proper mergings of a $3$-antichain and an
$n$-chain is (up to a shift) given by \cite{sloane}*{A085465}. This sequence counts the number of 
monotone $(n+1)$-colorings of the complete bipartite digraph $\vec{K}_{3,3}$, and was first mentioned in 
\cite{jovovic04antichains}, in a more general form. But let us first recall some definitions.

A \emph{directed graph} (digraph for short) is a tuple $(V,\vec{E})$, where $V$ is a set of 
\emph{vertices}, and $\vec{E}\subseteq V\times V$ is a set of \emph{directed edges}. A directed 
edge $(v_{1},v_{2})\in\vec{E}$ is to be understood as being directed from 
$v_{1}$ to $v_{2}$. We call a digraph $(V,\vec{E})$ \emph{complete bipartite} if we can partition $V$ 
into two disjoint sets $V_{1}$ and $V_{2}$ such that 
$\vec{E}=V_{1}\times V_{2}$. In the case 
$\lvert V_{1}\rvert=m_{1}$ and $\lvert V_{2}\rvert=m_{2}$, we simply write $\vec{K}_{m_{1},m_{2}}$
instead of $(V,\vec{E})$. 

A \emph{$k$-coloring of $(V,\vec{E})$} is a map $\gamma:V\rightarrow\{1,2,\ldots,k\}$. A $k$-coloring 
$\gamma$ is called \emph{monotone} if $(v_{1},v_{2})\in\vec{E}$ implies 
$\gamma(v_{1})\leq\gamma(v_{2})$. See \prettyref{fig:complete_bipartite_4_4} for an illustration.
\begin{figure}
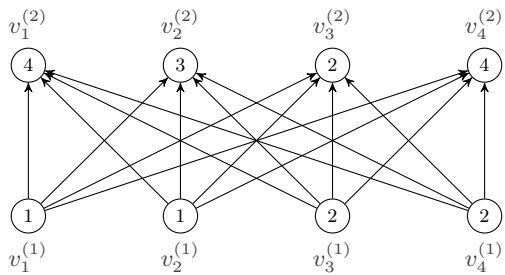

	\centering
	\exampleFourColoring
	\caption{The complete bipartite graph $\vec{K}_{4,4}$, and a monotone $4$-coloring.}
	\label{fig:complete_bipartite_4_4}
\end{figure}
As already mentioned in the beginning of this section, there exists a general formula for the number 
of monotone $k$-colorings of $\vec{K}_{m_{1},m_{2}}$.

\begin{proposition}[\cite{jovovic04antichains}*{Proposition~4.5}]
  \label{prop:monotone_colorings}
	For every $k,m_{1},m_{2}\in\mathbb{N}$, let $\eta_{k}(\vec{K}_{m_{1},m_{2}})$ denote the
	number of monotone $k$-colorings of the complete bipartite digraph $\vec{G}_{m_{1},m_{2}}$.
	Then,
	\begin{align*}
		\eta_{k}(\vec{K}_{m_{1},m_{2}}) = 
		  \sum_{i=1}^{k}{\Bigl((k+1-i)^{m_{1}}-(k-i)^{m_{1}}\Bigr)\cdot i^{m_{2}}}.
	\end{align*}
	Equivalently,
	\begin{align*}
		\eta_{k}(\vec{K}_{m_{1},m_{2}}) = 
		  \sum_{i=1}^{k}{\Bigl((k+1-i)^{m_{2}}-(k-i)^{m_{2}}\Bigr)\cdot i^{m_{1}}}.
	\end{align*}
\end{proposition}

In the light of this proposition, we notice immediately that \eqref{eq:formula_antichain_chain} 
corresponds to $\eta_{n+1}(\vec{K}_{m,m})$. Let $\Gamma_{n+1}(\vec{K}_{m,m})$ denote the set of 
monotone $(n+1)$-colorings of $\vec{K}_{m,m}$. 

\subsection{A Bijection between Monotone Colorings and Proper Mergings of an Antichain and a Chain}
  \label{sec:bijection_antichain_chain}
Let $\ac_{(m)}=(A,=)$, with $A=\{a_{1},a_{2},\ldots,a_{m}\}$, denote an $m$-antichain, and let 
$\cc_{(n)}=(C,\leq)$, with $C=\{c_{1},c_{2},\ldots,c_{n}\}$, denote an $n$-chain, where the order is
indicated by the indices. Since $\vec{K}_{m,m}$ consists of two independent sets of size $m$, it is 
obvious to relate these independent sets to the antichain $\ac_{(m)}$. We recall from 
\prettyref{sec:intents_extents} that the contraordinal scale of $\cc_{(n)}$ has precisely $(n+1)$ 
extents. Since we consider monotone $(n+1)$-colorings of $\vec{K}_{m,m}$, it is quite evident to 
relate the color of a vertex in $\vec{K}_{m,m}$ to an extent of $\CC(\cc_{(n)})$. 

\begin{definition}
  \label{def:bijection_antichain_chain}
	Let $\gamma\in\Gamma_{n+1}(\vec{K}_{m,m})$ be a monotone $(n+1)$-coloring of $\vec{K}_{m,m}$.
	Let the vertex set $V$ of $\vec{K}_{m,m}$ be partitioned into sets 
	$V_{1}=\{v_{1}^{(1)},v_{2}^{(1)},\ldots,v_{m}^{(1)}\}$ and 
	$V_{2}=\{v_{1}^{(2)},v_{2}^{(2)},\ldots,v_{m}^{(2)}\}$. Let $\ac_{(m)}=(A,=)$ denote an 
	$m$-antichain with ground set $A=\{a_{1},a_{2},\ldots,a_{m}\}$, and let
	$\cc_{(n)}=(C,\leq)$ denote an $n$-chain with ground set $C=\{c_{1},c_{2},\ldots,c_{n}\}$, 
	where the order is indicated by the indices. Define relations $R_{\gamma}\subseteq A\times C$,
	and $S_{\gamma}\subseteq C\times A$ as 
	\begin{align}
		a_{i}\;R_{\gamma}\;c_{j} & \quad\text{if and only if}\quad 
		  \gamma\bigl(v_{i}^{(1)}\bigr)=k\;\text{and}\;n+2-k\leq j\leq n,\\
		c_{j}\;S_{\gamma}\;a_{i} & \quad\text{if and only if}\quad 
		  \gamma\bigl(v_{i}^{(2)}\bigr)=k\;\text{and}\;1\leq j\leq n+1-k,
	\end{align}
	for all $1\leq i\leq m$, and $1\leq j\leq n$.
\end{definition}
This means that the row $\{a_{i}\}^{R}$ corresponds to the $(n+2-k)$-th intent of 
$\CL\bigl(\CC(\cc_{(n)})\bigr)$ read from bottom to top if and only if the vertex $v_{i}^{(1)}$ has 
color $k$. Similarly, the column $\{a_{i}\}^{S}$ corresponds to the $(n+2-k)$-th extent of 
$\CL\bigl(\CC(\cc_{(n)})\bigr)$ read from bottom to top if and only if the vertex $v_{i}^{(2)}$ has
color $k$. See \prettyref{fig:merging_from_digraph_cxt} for an illustration. 

\begin{figure}
	\centering
	\begin{tikzpicture}
		\draw(0,0) node{\exampleFourRTable};
		\draw(5,0) node{\exampleFourSTable};
	\end{tikzpicture}
	\caption{The relations $R$ and $S$ induced by the monotone coloring of $\vec{K}_{4,4}$ 
	  depicted in \prettyref{fig:complete_bipartite_4_4}.}
	\label{fig:merging_from_digraph_cxt}
\end{figure}

\begin{lemma}
  \label{lem:create_merging_antichain_chain}
	The relations $R_{\gamma}$ and $S_{\gamma}$ from \prettyref{def:bijection_antichain_chain}
	form a proper merging of $\ac_{(m)}$ and $\cc_{(n)}$.
\end{lemma}
\begin{proof}
	Let $\ac_{(m)}=(A,=)$ be an antichain, and denote by $\KK(\ac_{(m)})$ the corresponding formal
	context $(A,A,=)$. We initiate the proof with the investigation of the intents and extents 
	of the contraordinal scale $\CC(\ac_{(m)})=(A,A,\neq)$. Since $\ac_{(m)}$ is an antichain, we 
	can write the cross-table of $\KK(\ac_{(m)})$ in such a way that there are only crosses on
	the main diagonal. It is immediate that we can write the cross-table of $\CC(\ac_{(m)})$ in 
	such a way that there are crosses in every cell which is not on the main diagonal. It is 
	well-known that the concept lattice $\CL\bigl(\CC(\ac_{(m)})\bigr)$ is isomorphic to the 
	Boolean lattice with $2^{m}$ elements. See \prettyref{fig:boolean_context} for an illustration.
	This implies that every subset of $A$ is an intent and an extent of $\CC(\ac_{(m)})$. 

	It is immediate from \prettyref{def:bijection_antichain_chain} that every row of $R_{\gamma}$
	corresponds to an intent of $\CC(\cc_{(n)})$, and that every column of $S_{\gamma}$ corresponds to
	an extent of $\CC(\cc_{(n)})$. With the previous reasoning, this implies that $R_{\gamma}$ is a 
	bond from $\CC(\ac_{(m)})$ to $\CC(\cc_{(n)})$ and $S_{\gamma}$ is a bond from $\CC(\cc_{(n)})$ to 
	$\CC(\ac_{(m)})$. Hence, conditions (1) and (2) of \prettyref{prop:classification_mergings} are 
	satisfied.

	We need to show conditions (3) and (4) of \prettyref{prop:classification_mergings}, namely 
	that $R_{\gamma}\circ S_{\gamma}$ is contained in the order relation of $\ac_{(m)}$, and that 
	$S_{\gamma}\circ R_{\gamma}$ is contained in the order relation of $\cc_{(n)}$. Let 
	$a_{i},a_{j}\in A$ satisfy $(a_{i},a_{j})\in R_{\gamma}\circ S_{\gamma}$, and let 
	$\gamma\bigl(v_{i}^{(1)}\bigr)=l_{1},\gamma\bigl(v_{j}^{(2)}\bigr)=l_{2}$. This means that there 
	is an element $c_{k}\in C$ with $n+2-l_{1}\leq k\leq n$ and $1\leq k\leq n+1-l_{2}$. Since 
	$\gamma$ is a monotone coloring, we know that $l_{1}\leq l_{2}$. We obtain
	\begin{align*}
		k\leq n+1-l_{2}\leq n+1-l_{1}<n+2-l_{1}\leq k,
	\end{align*}
	and thus $k<k$, which is a contradiction. Hence, $R_{\gamma}\circ S_{\gamma}=\emptyset$, 
	which proves condition (3). Let now, in turn, $c_{i},c_{j}\in C$ satisfy 
	$(c_{i},c_{j})\in S_{\gamma}\circ R_{\gamma}$. This means, there must be some $a_{k}\in A$
	such that the colors $\gamma\bigl(v_{k}^{(1)}\bigr)=l_{1}$, and 
	$\gamma\bigl(v_{k}^{(2)}\bigr)=l_{2}$ satisfy $n+2-j\leq l_{1}$ and $l_{2}\leq n+1-i$. Since
	$\gamma$ is a monotone coloring, we know that $l_{1}\leq l_{2}$, which implies
	\begin{align*}
		n+2-j\leq l_{1}\leq l_{2}\leq n+1-i.
	\end{align*}
	Hence, $i<j$, and $S_{\gamma}\circ R_{\gamma}$ is contained in the order relation of
	$\cc_{(n)}$ as desired for condition (4).

	It remains to show that $R_{\gamma}\cap S_{\gamma}^{-1}=\emptyset$. Assume the opposite,
	and let $(a_{i},c_{j})\in R_{\gamma}\cap S_{\gamma}^{-1}$. Let 
	$\gamma\bigl(v_{i}^{(1)}\bigr)=l_{1}$, and $\gamma\bigl(v_{i}^{(2)}\bigr)=l_{2}$. Hence,
	\begin{align*}
		n+2-l_{1}\leq j\leq n+1-l_{2},
	\end{align*}
	which implies $l_{2}<l_{1}$. This is a contradiction to $\gamma$ being a monotone coloring.
\end{proof} 

\begin{figure}
	\centering
	\exampleFourPoset
	\caption{The proper merging of a $4$-antichain $\ac_{(4)}$ and a $3$-chain $\cc_{(3)}$ defined by 
	  the relations given in \prettyref{fig:merging_from_digraph_cxt}. The green nodes represent 
	  $\ac_{(4)}$, and the black nodes represent $\cc_{(3)}$.}
	\label{fig:merging_from_digraph}
\end{figure}

\prettyref{fig:merging_from_digraph} shows the poset corresponding to the proper merging depicted in
\prettyref{fig:merging_from_digraph_cxt}. We can conclude the following theorem.

\begin{theorem}
  \label{thm:bijection_antichain_chain}
	Let $\Gamma_{n+1}(\vec{K}_{m,m})$ denote the set of monotone $(n+1)$-colorings of $\vec{K}_{m,m}$.
	Let $\mathfrak{A\!C}_{m,n}^{\bullet}$ denote the set of proper mergings of an 
	$m$-antichain and an $n$-chain. Then, the correspondence described in 
	\prettyref{def:bijection_antichain_chain} is a bijection between 
	$\Gamma_{n+1}(\vec{K}_{m,m})$ and $\mathfrak{A\!C}_{m,n}^{\bullet}$. 
\end{theorem}
\begin{proof}
	It follows immediately from \prettyref{lem:create_merging_antichain_chain} that each monotone 
	$(n+1)$-coloring of $\vec{K}_{m,m}$ induces a proper merging of an $m$-antichain and an $n$-chain. 

	Let $\ac_{(m)}=(A,=)$ be an $m$-antichain, where $A=\{a_{1},a_{2},\ldots,a_{m}\}$, and let 
	$\cc_{(n)}=(C,\leq)$ be an $n$-chain, where $C=\{c_{1},c_{2},\ldots,c_{n}\}$ and the ordering
	is induced by the indices. Let $(R,S)$ be a proper merging of $\ac_{(m)}$ and $\cc_{(n)}$. 
	Consider the complete bipartite graph $\vec{K}_{m,m}$ and let its vertex set be partitioned 
	into $V_{1}$ and $V_{2}$, with $V_{1}=\{v_{1}^{(1)},v_{2}^{(1)},\ldots,v_{m}^{(1)}\}$ and
	$V_{2}=\{v_{1}^{(2)},v_{2}^{(2)},\ldots,v_{m}^{(2)}\}$. Define a coloring $\gamma_{(R,S)}$ 
	of $\vec{K}_{m,m}$ via
	\begin{multline*}
		\gamma_{(R,S)}\bigl(v_{i}^{(1)}\bigr)=k\quad\text{if and only if}\\
		\shoveright{a_{i}\;R\;c_{j}\;\text{for all}\;j\in\{n+2-k,n+3-k,\ldots,n\},\quad\text{and}}\\
		\shoveleft{\gamma_{(R,S)}\bigl(v_{i}^{(2)}\bigr)=k\quad\text{if and only if}}\\
		c_{j}\;S\;a_{i}\;\text{for all}\;j\in\{1,2,\ldots,n+1-k\},
	\end{multline*}
	for all $i\in\{1,2,\ldots,m\}$. Since $R$ is a bond from $\CC(\ac_{(m)})$ to $\CC(\cc_{(n)})$, 
	every subset of $V_{1}$ can be colored with color $k$, for $1\leq k\leq n+1$. (Every subset of $A$ 
	is an extent of $\CC(\ac_{(m)})$, and the set $\{c_{n+2-k},c_{n+3-k},\ldots,c_{n}\}$ is an 
	intent of $\CC(\cc_{(n)})$ for every $k\in\{1,2,\ldots,n+1\}$.) Since $S$ is a bond from 
	$\CC(\cc_{(n)})$ to $\CC(\ac_{(m)})$, the same property holds for $V_{2}$. Hence, $\gamma_{(R,S)}$ 
	is an $(n+1)$-coloring of $\vec{K}_{m,m}$.

	Let $v_{i}^{(1)}\in V_{1}$ and $v_{j}^{(2)}\in V_{2}$, with 
	$\gamma_{(R,S)}\bigl(v_{i}^{(1)}\bigr)=l_{1}$ and 
	$\gamma_{(R,S)}\bigl(v_{j}^{(2)}\bigr)=l_{2}$. By definition, it follows that 
	$a_{i}\;R\;c_{k_{1}}$ for all $k_{1}\in\{n+2-l_{1},n+3-l_{1},\ldots,n\}$, and 
	$c_{k_{2}}\;S\;a_{j}$ for all $k_{2}\in\{1,2,\ldots,n+1-l_{2}\}$. Assume that $l_{1}>l_{2}$. 
	Hence, there exists a $k\in\{1,2,\ldots,n\}$ with $a_{i}\;R\;c_{k}$ and $c_{k}\;S\;a_{j}$. 
	Since $R\circ S$ is contained in the order relation of $\ac_{(m)}$, it follows that 
	$a_{i}=a_{j}$. This means in particular that $a_{i}\;R\;c_{n+2-l_{1}}$ and 
	$c_{n+1-l_{2}}\;S\;a_{i}$, and thus $(c_{n+1-l_{2}},c_{n+2-l_{1}})\in S\circ R$. If 
	$l_{1}=l_{2}+1$, then $c_{n+2-l_{1}}=c_{n+1-l_{2}}$, which is a contradiction to 
	$R\cap S^{-1}=\emptyset$. If $l_{1}>l_{2}+1$, then $c_{n+2-l_{1}}<c_{n+1-l_{2}}$, which is a 
	contradiction to $S\circ R$ being contained in the order relation of $\cc_{(n)}$. Thus, 
	$\gamma_{(R,S)}$ is a monotone coloring of $\vec{K}_{m,m}$ with at most $n+1$ colors.
\end{proof}

An extensive illustration of this bijection can be found in \prettyref{app:bijection_antichain_chain}. 

\begin{proof}[Proof of \prettyref{thm:enumeration_antichain_chain}]
	This follows immediately from \prettyref{thm:bijection_antichain_chain} and 
	\prettyref{prop:monotone_colorings}. 
\end{proof}

\begin{remark}
	Let $\gamma,\delta\in\Gamma_{n+1}(\vec{K}_{m,m})$, and let $V$ denote the vertex set of
	$\vec{K}_{m,m}$. Define a partial order $\leq$ as
	\begin{align*}
		\gamma\leq\delta\quad\text{if and only if}\quad \gamma(v)\leq\delta(v),
	\end{align*}
	for all vertices $v\in V$. Consider the partition $V=V_{1}\cup V_{2}$. Let
	$(R_{\gamma},S_{\gamma})$, and $(R_{\delta},S_{\delta})$ denote the proper mergings associated to
	$\gamma$ respectively to $\delta$ in the sense of \prettyref{def:bijection_antichain_chain}. 
	Suppose that $(R_{\gamma},S_{\gamma})\preceq(R_{\delta},S_{\delta})$, and hence by definition
	$R_{\gamma}\subseteq R_{\delta}$, and $S_{\gamma}\supseteq S_{\delta}$. This implies  
	$\gamma(v)\leq\delta(v)$ if $v\in V_{1}$, and $\gamma(v)\leq\delta(v)$ if $v\in V_{2}$, and
	hence $\gamma\leq\delta$. This means that the bijection described in 
	\prettyref{thm:bijection_antichain_chain} is indeed an isomorphism between the lattices
	$(\mathfrak{A\!C}_{m,n}^{\bullet},\preceq)$ and $(\Gamma_{n+1}(\vec{K}_{m,m}),\leq)$. 
\end{remark}

\subsection{Counting Galois Connections between Boolean Lattices and Chains}
  \label{sec:galois_connections_boolean_chain}
Similarly to \prettyref{sec:galois_connections_chains}, we can exploit the bijection described in
\prettyref{thm:bijection_antichain_chain} in order to count the Galois connections between
chains with $n+1$ elements and Boolean lattices with $2^{m}$ elements. 
\prettyref{thm:galois_connections} states that every such Galois connection can be described as a 
dual bond from $(C,C,<)$ to $(A,A,\neq)$, where $A=\{a_{1},a_{2},\ldots,a_{m})$ and 
$C=\{c_{1},c_{2},\ldots,c_{n}\}$. Since every extent of $(A,A,\neq)$ is also an intent of $(A,A,\neq)$
every dual bond from $(C,C,<)$ to $(A,A,\neq)$ corresponds to a bond from $(C,C,<)$ to $(A,A,\neq)$. 
By definition, each such bond corresponds to a proper merging of $(A,=)$ and $(C,\leq)$, which is 
of the form $(\emptyset,\cdot)$. It follows immediately from \prettyref{def:bijection_antichain_chain}
that each such proper merging corresponds to a monotone coloring of $\vec{K}_{m,m}$, where the 
vertices in $V_{1}$ all have color $1$. Hence, each vertex in $V_{2}$ can take every color 
$k\in\{1,2,\ldots,n+1\}$. Thus, we can conclude the following proposition.

\begin{proposition}
  \label{prop:galois_connections_boolean_chain}
	Let $B_{m}$ denote the Boolean lattice with $2^{m}$ elements, and let $\cc_{(n+1)}$ denote a
	chain with $n+1$ elements. The number of Galois connections between $B_{m}$ and $\cc_{(n+1)}$ is
	$(n+1)^{m}$.
\end{proposition}

An extensive illustration of this proposition can be found in 
\prettyref{app:bijection_galois_bonds_antichain_chain}.

\section*{Acknowledgements}
The author wants to thank Christian Krattenthaler for many helpful suggestions on the proof of 
\prettyref{thm:enumeration_antichains}, in particular for finding the generating function.

\bibliography{../../literature}

\begin{bibdiv}
\begin{biblist}

\bib{conexp}{misc}{
      author={Borchmann, Daniel},
       title={{\texttt{conexp-clj} -- An Extensive Tool for Computations in
  Formal Concept Analysis.}},
        note={\url{http://daniel.kxpq.de/math/conexp-clj/}},
}

\bib{chen09pairs}{article}{
      author={Chen, William Y.~C.},
      author={Pang, Sabrina X.~M.},
      author={Qu, Ellen X.~Y.},
      author={Stanley, Richard},
       title={{Pairs of Noncrossing Free Dyck Paths and Noncrossing
  Partitions}},
        date={2007},
     journal={Discrete Mathematics},
      volume={309},
       pages={2834\ndash 2838},
}

\bib{ganter11merging}{inproceedings}{
      author={Ganter, Bernhard},
      author={Meschke, Christian},
      author={M{\"u}hle, Henri},
       title={{Merging Ordered Sets}},
        date={2011},
   booktitle={{Proceedings of the 9th International Conference on Formal
  Concept Analysis}},
      editor={J{\"a}schke, Robert},
      editor={Valtchev, Petko},
      series={Lecture Notes in Artificial Intelligence},
      volume={6628},
   publisher={Springer},
       pages={183\ndash 203},
}

\bib{ganter99formal}{book}{
      author={Ganter, Bernhard},
      author={Wille, Rudolf},
       title={{Formal Concept Analysis: Mathematical Foundations}},
   publisher={Springer},
     address={Heidelberg},
        date={1999},
}

\bib{jovovic04antichains}{article}{
      author={Jovovi\'c, Vladeta},
      author={Kilibarda, Goran},
       title={{Antichains of Multisets}},
        date={2004},
     journal={Journal of Integer Sequences},
      volume={7},
}

\bib{macdonald95symmetric}{book}{
      author={Macdonald, Ian~G.},
       title={{Symmetric Functions and Hall Polynomials}},
   publisher={Oxford University Press},
     address={Oxford},
        date={1995},
}

\bib{macmahon97memoir}{article}{
      author={MacMahon, Percy~A.},
       title={{Memoir on the Theory of the Partitions of Numbers -- Part 1}},
        date={1897},
     journal={Philosophical Transactions of the Royal Society of London (A)},
      volume={187},
       pages={619\ndash 673},
}

\bib{macmahon16combinatory}{book}{
      author={MacMahon, Percy~A.},
       title={{Combinatory Analysis, Vol. 2}},
   publisher={Cambridge University Press},
     address={Cambridge},
        date={1916},
}

\bib{sloane}{misc}{
      author={Sloane, Neil J.~A.},
       title={{The Online Encyclopedia of Integer Sequences}},
        note={\url{http://www.research.att.com/\~njas/sequences/}},
}

\bib{stanley01enumerative}{book}{
      author={Stanley, Richard~P.},
       title={{Enumerative Combinatorics, Vol. 2}},
   publisher={Cambridge University Press},
     address={Cambridge},
        date={2001},
}

\bib{wilf06generating}{book}{
      author={Wilf, Herbert~S.},
       title={{generatingfunctionology}},
   publisher={A. K. Peters, Ltd.},
     address={Natick},
        date={2006},
}

\bib{wille82restructuring}{article}{
      author={Wille, Rudolf},
       title={{Restructuring Lattice Theory: An Approach Based on Hierarchies
  of Concepts}},
        date={1982},
     journal={Ordered Sets},
       pages={314\ndash 339},
}

\end{biblist}
\end{bibdiv}

\clearpage

\appendix

\section{Illustration of \prettyref{thm:bijection_chains}, $m=n=2$}
  \label{app:bijection_chains}
\begin{longtable}{c|cc|c}
	$\pi\in\text{PP}_{2,2}^{(2)}$ & $R_{\pi}$ & $S_{\pi}$ 
		& $(C_{1}\cup C_{2},\leq_{R_{\pi},S_{\pi}})$\\
	\hline
	\plp{0}{0}{0}{0} & \rTabOne & \sTabOne & \lOne{0}{1} \\
	\hline
	\plp{1}{0}{0}{0} & \rTabOne & \sTabTwo & \raisebox{.15cm}{\lTwo{0}{1}} \\
	\hline
	\plp{1}{0}{1}{0} & \rTabOne & \sTabThree & \raisebox{.15cm}{\lFour{0}{1}} \\
	\hline
	\plp{1}{1}{0}{0} & \rTabOne & \sTabFour & \raisebox{.15cm}{\lThree{0}{1}} \\
	\hline
	\plp{1}{1}{1}{0} & \rTabOne & \sTabFive & \raisebox{.3cm}{\lSix{0}{1}} \\
	\hline
	\plp{1}{1}{1}{1} & \rTabOne & \sTabSix & \raisebox{.3cm}{\lNine{0}{1}} \\
	\hline
	\plp{2}{0}{0}{0} & \rTabTwo & \sTabTwo & \lFive{0}{1} \\
	\hline
	\plp{2}{0}{1}{0} & \rTabTwo & \sTabThree & \raisebox{.15cm}{\lEight{0}{1}} \\
	\hline
	\plp{2}{0}{2}{0} & \rTabThree & \sTabThree & \lTwelve{0}{1} \\
	\hline
	\plp{2}{1}{0}{0} & \rTabTwo & \sTabFour & \raisebox{.15cm}{\lSeven{0}{1}} \\
	\hline
	\plp{2}{1}{1}{0} & \rTabTwo & \sTabFive & \raisebox{.3cm}{\lTen{0}{1}}\\
	\hline
	\plp{2}{1}{1}{1} & \rTabTwo & \sTabSix & \raisebox{.3cm}{\reflectbox{\lSix{180}{1}}} \\
	\hline
	\plp{2}{1}{2}{0} & \rTabThree & \sTabFive & \raisebox{.15cm}{\reflectbox{\lEight{180}{1}}} \\
	\hline
	\plp{2}{1}{2}{1} & \rTabThree & \sTabSix & \raisebox{.15cm}{\reflectbox{\lFour{180}{1}}} \\
	\hline
	\plp{2}{2}{0}{0} & \rTabFour & \sTabFour & \lEleven{0}{1} \\
	\hline
	\plp{2}{2}{1}{0} & \rTabFour & \sTabFive & \raisebox{.15cm}{\reflectbox{\lSeven{180}{1}}} \\
	\hline
	\plp{2}{2}{1}{1} & \rTabFour & \sTabSix & \raisebox{.15cm}{\reflectbox{\lThree{180}{1}}} \\
	\hline
	\plp{2}{2}{2}{0} & \rTabFive & \sTabFive & \lFive{180}{1} \\
	\hline
	\plp{2}{2}{2}{1} & \rTabFive & \sTabSix & \raisebox{.15cm}{\reflectbox{\lTwo{180}{1}}} \\
	\hline
	\plp{2}{2}{2}{2} & \rTabSix & \sTabSix & \raisebox{-.25cm}{\lOne{180}{1}} \\
\end{longtable}

\section{Illustration of \prettyref{rem:generalized_bijection_chains}, $n_{1}=n_{2}=n_{3}=1$}
  \label{app:generalized_bijection_chains}
\begin{longtable}{c|ccc|c}
	Collection & $\pi_{1}$ & $\pi_{2}$ & $\pi_{3}$ & Proper Merging\\
	\hline
	\solPar{1}{0}{0} & \planeOne{1} & \planeTwo{0} & \planeThree{0} & \multirow{2}{.75cm}{\uEight{0}{1}}\\
	\solPar{2}{0}{0} & \planeOne{2} & \planeTwo{0} & \planeThree{0} & \\
	\hline
	\solPar{0}{1}{0} & \planeOne{0} & \planeTwo{1} & \planeThree{0} & \multirow{2}{.75cm}{\uFour{0}{1}}\\
	\solPar{0}{2}{0} & \planeOne{0} & \planeTwo{2} & \planeThree{0} & \\
	\hline
	\solPar{0}{0}{1} & \planeOne{0} & \planeTwo{0} & \planeThree{1} & \multirow{2}{.75cm}{\uOne{180}{1}}\\
	\solPar{0}{0}{2} & \planeOne{0} & \planeTwo{0} & \planeThree{2} & \\
	\hline
	\solPar{1}{1}{0} & \planeOne{1} & \planeTwo{1} & \planeThree{0} & \raisebox{.45cm}{\uSix{0}{1}}\\
	\hline
	\solPar{1}{0}{1} & \planeOne{1} & \planeTwo{0} & \planeThree{1} & \reflectbox{\raisebox{.45cm}{\uFive{180}{1}}}\\
	\hline
	\solPar{0}{1}{1} & \planeOne{0} & \planeTwo{1} & \planeThree{1} & \raisebox{.45cm}{\uSeven{180}{1}}\\
	\hline
	\solPar{1}{1}{1} & \planeOne{1} & \planeTwo{1} & \planeThree{1} & \raisebox{.7cm}{\uTen{0}{1}}\\
	\hline
	\solPar{0}{1}{2} & \planeOne{0} & \planeTwo{1} & \planeThree{2} & \raisebox{.45cm}{\uThree{180}{1}}\\
	\hline
	\solPar{0}{2}{1} & \planeOne{0} & \planeTwo{2} & \planeThree{1} & \raisebox{.45cm}{\uNine{0}{1}}\\
	\hline
	\solPar{1}{0}{2} & \planeOne{1} & \planeTwo{0} & \planeThree{2} & \reflectbox{\raisebox{.45cm}{\uTwo{180}{1}}}\\
	\hline
	\solPar{1}{2}{0} & \planeOne{1} & \planeTwo{2} & \planeThree{0} & \raisebox{.45cm}{\uTwo{0}{1}}\\
	\hline
	\solPar{2}{0}{1} & \planeOne{2} & \planeTwo{0} & \planeThree{1} & \reflectbox{\raisebox{.45cm}{\uNine{180}{1}}}\\
	\hline
	\solPar{2}{1}{0} & \planeOne{2} & \planeTwo{1} & \planeThree{0} & \raisebox{.45cm}{\uThree{0}{1}}\\
	\hline
	\solPar{1}{1}{2} & \planeOne{1} & \planeTwo{1} & \planeThree{2} & \reflectbox{\raisebox{.45cm}{\uSix{180}{1}}}\\
	\hline
	\solPar{1}{2}{1} & \planeOne{1} & \planeTwo{2} & \planeThree{1} & \raisebox{.45cm}{\uFive{0}{1}}\\
	\hline
	\solPar{2}{1}{1} & \planeOne{2} & \planeTwo{1} & \planeThree{1} & \reflectbox{\raisebox{.45cm}{\uSeven{0}{1}}}\\
	\hline
	\solPar{1}{2}{2} & \planeOne{1} & \planeTwo{2} & \planeThree{2} & \multirow{2}{.75cm}{\uEight{180}{1}}\\
	\solPar{0}{2}{2} & \planeOne{0} & \planeTwo{2} & \planeThree{2} & \\
	\hline
	\solPar{2}{1}{2} & \planeOne{2} & \planeTwo{1} & \planeThree{2} & \multirow{2}{.75cm}{\uFour{180}{1}}\\
	\solPar{2}{0}{2} & \planeOne{2} & \planeTwo{0} & \planeThree{2} & \\
	\hline
	\solPar{2}{2}{1} & \planeOne{2} & \planeTwo{2} & \planeThree{1} & \multirow{2}{.75cm}{\uOne{0}{1}}\\
	\solPar{2}{2}{0} & \planeOne{2} & \planeTwo{2} & \planeThree{0} & \\
\end{longtable}

\section{Illustration of \prettyref{prop:galois_connections_chains}, $m=n=2$}
  \label{app:bijection_galois_bonds}
\begin{longtable}{c|c|cc}
	$\pi\in\text{PP}_{2,2}^{(1)}$ & $T_{\pi}$ & $\psi_{T_{\pi}}$ & $\varphi_{T_{\pi}}$ \\
	\hline
	\plp{0}{0}{0}{0} & \tTabOne & \GalPsi{b1}{a3}{b2}{a3}{b3}{a3} 
	  & \GalPhi{a1}{b3}{a2}{b3}{a3}{b3} \\
	\hline
	\plp{1}{0}{0}{0} & \tTabTwo & \GalPsi{b1}{a3}{b2}{a3}{b3}{a2} 
	  & \GalPhi{a1}{b3}{a2}{b3}{a3}{b2} \\
	\hline
	\plp{1}{0}{1}{0} & \tTabThree & \GalPsi{b1}{a3}{b2}{a3}{b3}{a1} 
	  & \GalPhi{a1}{b3}{a2}{b2}{a3}{b2} \\
	\hline
	\plp{1}{1}{0}{0} & \tTabFour & \GalPsi{b1}{a3}{b2}{a2}{b3}{a2} 
	  & \GalPhi{a1}{b3}{a2}{b3}{a3}{b1} \\
	\hline
	\plp{1}{1}{1}{0} & \tTabFive & \GalPsi{b1}{a3}{b2}{a2}{b3}{a1} 
	  & \GalPhi{a1}{b3}{a2}{b2}{a3}{b1} \\
	\hline
	\plp{1}{1}{1}{1} & \tTabSix & \GalPsi{b1}{a3}{b2}{a1}{b3}{a1} 
	  & \GalPhi{a1}{b3}{a2}{b1}{a3}{b1} \\
\end{longtable}

\section{Illustration of \prettyref{thm:bijection_antichain_chain}, $m=n=2$}
  \label{app:bijection_antichain_chain}
\begin{longtable}{c|cc|c}
	$\gamma\in\Gamma_{3}(\vec{K}_{2,2})$ & $R_{\gamma}$ & $S_{\gamma}$ 
			& $(A\cup C,\leq_{R_{\gamma},S_{\gamma}})$ \\
	\hline
	\bip{1}{1}{1}{1} & \rtabOne & \stabOne & \nOne{180}{1} \\
	\hline
	\bip{1}{1}{1}{2} & \rtabOne & \stabTwo & \nThree{180}{1} \\
	\hline
	\bip{1}{1}{1}{3} & \rtabOne & \stabThree & \nSix{180}{1} \\
	\hline
	\bip{1}{1}{2}{1} & \rtabOne & \stabFour & \nTwo{180}{1} \\
	\hline
	\bip{1}{1}{2}{2} & \rtabOne & \stabFive & \nFive{180}{1} \\
	\hline
	\bip{1}{2}{2}{2} & \rtabTwo & \stabFive & \nTen{180}{1} \\
	\hline
	\bip{2}{1}{2}{2} & \rtabThree & \stabFive & \nSeven{180}{1} \\
	\hline
	\bip{2}{2}{2}{2} & \rtabFour & \stabFive & \nThirteen{0}{1} \\
	\hline
	\bip{1}{1}{2}{3} & \rtabOne & \stabSix & \nNine{180}{1} \\
	\hline
	\bip{1}{2}{2}{3} & \rtabTwo & \stabSix & \nFifteen{0}{1} \\
	\hline
	\bip{2}{1}{2}{3} & \rtabThree & \stabSix & \nTwelve{0}{1} \\
	\hline
	\bip{2}{2}{2}{3} & \rtabFour & \stabSix & \nTen{0}{1} \\
	\hline
	\bip{1}{1}{3}{1} & \rtabOne & \stabSeven & \nFour{180}{1} \\
	\hline
	\bip{1}{1}{3}{2} & \rtabOne & \stabEight & \nEight{180}{1} \\
	\hline
	\bip{1}{2}{3}{2} & \rtabTwo & \stabEight & \nTwelve{180}{1} \\
	\hline
	\bip{2}{1}{3}{2} & \rtabThree & \stabEight & \nEleven{0}{1} \\
	\hline
	\bip{2}{2}{3}{2} & \rtabFour & \stabEight & \nSeven{0}{1} \\
	\hline
	\bip{1}{1}{3}{3} & \rtabOne & \stabNine & \nFourteen{0}{1} \\
	\hline
	\bip{1}{2}{3}{3} & \rtabTwo & \stabNine & \nNine{0}{1} \\
	\hline
	\bip{1}{3}{3}{3} & \rtabFive & \stabNine & \nSix{0}{1} \\
	\hline
	\bip{2}{1}{3}{3} & \rtabThree & \stabNine & \nEight{0}{1} \\
	\hline
	\bip{2}{2}{3}{3} & \rtabFour & \stabNine & \nFive{0}{1} \\
	\hline
	\bip{2}{3}{3}{3} & \rtabSix & \stabNine & \nThree{0}{1} \\
	\hline
	\bip{3}{1}{3}{3} & \rtabSeven & \stabNine & \nFour{0}{1} \\
	\hline
	\bip{3}{2}{3}{3} & \rtabEight & \stabNine & \nTwo{0}{1} \\
	\hline
	\bip{3}{3}{3}{3} & \rtabNine & \stabNine & \nOne{0}{1} \\
\end{longtable}

\section{Illustration of \prettyref{prop:galois_connections_boolean_chain}, $m=n=2$}
  \label{app:bijection_galois_bonds_antichain_chain}
\begin{longtable}{c|c|cc}
	$\gamma\in\Gamma_{3}(\vec{K}_{2,2}), \gamma(V_{1})\equiv 1$ & $S_{\gamma}$ 
		& $\psi_{S_{\gamma}}$ & $\varphi_{S_{\gamma}}$ \\
	\hline
	\bip{1}{1}{1}{1} & \stabOne & \galPsi{b1}{a4}{b2}{a4}{b3}{a4} 
		& \galPhi{a1}{b3}{a2}{b3}{a3}{b3}{a4}{b3} \\
	\hline
	\bip{1}{1}{1}{2} & \stabTwo & \galPsi{b1}{a4}{b2}{a4}{b3}{a3} 
		& \galPhiTwo{a1}{b3}{a2}{b2}{a3}{b3}{a4}{b2} \\
	\hline
	\bip{1}{1}{1}{3} & \stabThree & \galPsiTwo{b1}{a4}{b2}{a3}{b3}{a3} 
		& \galPhi{a1}{b3}{a2}{b1}{a3}{b3}{a4}{b1} \\
	\hline
	\bip{1}{1}{2}{1} & \stabFour & \galPsi{b1}{a4}{b2}{a4}{b3}{a2} 
		& \galPhi{a1}{b3}{a2}{b3}{a3}{b2}{a4}{b2} \\
	\hline
	\bip{1}{1}{2}{2} & \stabFive & \galPsi{b1}{a4}{b2}{a4}{b3}{a1} 
		& \galPhiTwo{a1}{b3}{a2}{b2}{a3}{b2}{a4}{b2} \\
	\hline
	\bip{1}{1}{2}{3} & \stabSix & \galPsiTwo{b1}{a4}{b2}{a3}{b3}{a1} 
		& \galPhi{a1}{b3}{a2}{b1}{a3}{b2}{a4}{b1} \\
	\hline
	\bip{1}{1}{3}{1} & \stabSeven & \galPsi{b1}{a4}{b2}{a2}{b3}{a2} 
		& \galPhi{a1}{b3}{a2}{b3}{a3}{b1}{a4}{b1} \\
	\hline
	\bip{1}{1}{3}{2} & \stabEight & \galPsi{b1}{a4}{b2}{a2}{b3}{a1} 
		& \galPhiTwo{a1}{b3}{a2}{b2}{a3}{b1}{a4}{b1} \\
	\hline
	\bip{1}{1}{3}{3} & \stabNine & \galPsi{b1}{a4}{b2}{a1}{b3}{a1} 
		& \galPhi{a1}{b3}{a2}{b1}{a3}{b1}{a4}{b1} \\
\end{longtable}

\end{document}